%% file: Yue_Okten.tex
\documentclass[review,onefignum,onetabnum]{siamart190516}
\nolinenumbers

\def\E{{\rm E}\,}
\input{shared}

\ifpdf
\hypersetup{
pdftitle={Global activity scores},
pdfauthor={Ruilong Yue, Giray  \"{O}kten}
}
\fi

\begin{document}

\maketitle

% REQUIRED

\begin{abstract}
We introduce a new global sensitivity measure, the \textit{global activity scores}. 
The measure is based on finite differences of the underlying function, in contrast to several sensitivity measures in the literature that are based on derivatives of the function. We establish its theoretical connection with Sobol' sensitivity indices and demonstrate its performance through numerical examples. In these examples, we compare global activity scores with Sobol' sensitivity indices, derivative-based sensitivity measures, and activity scores. 
The results show that in the presence of additive noise or high variability, global activity scores provide more stable and reliable identification of influential variables than derivative-based measures and activity scores, which are more sensitive to noise. In noiseless settings, however, all three approaches yield comparable results.
\end{abstract}

% REQUIRED

% REQUIRED

\section{Introduction}
Global sensitivity analysis seeks to quantify how uncertainty in a model’s output can be 
attributed to uncertainties in its inputs. Such analyses are indispensable in modern 
computational modeling, where complex models frequently involve many parameters whose 
relative influence is not known a priori. Sensitivity information enables the modeler to 
identify the most influential inputs, reduce model complexity by fixing nonessential 
variables, guide data collection and experimental design, and improve the interpretability 
and robustness of the model. Applications of global sensitivity analysis span the natural 
sciences, engineering, social sciences, and mathematical sciences (Saltelli et al. \cite{saltelli2008global}).

Several approaches to global sensitivity analysis have been proposed in the literature. Variance-based Sobol' 
sensitivity indices provide a decomposition of output variance into contributions 
from individual inputs and their interactions (Saltelli et al. \cite{saltelli2010variance}, Sobol' \cite{sobol2001global}). Derivative-based global sensitivity measures assess the average influence of inputs using 
derivative information averaged over the parameter space (Kucherenko and Iooss \cite{kucherenko2014derivative}, Sobol' and Kucherenko \cite{sobol2009derivative}, Sobol' and Kucherenko \cite{sobol2010derivative}).
More recently, activity scores have been introduced as sensitivity measures based on the dominant directions of variation 
in a model (Constantine and Diaz \cite{constantine2017global}). These dominant directions are obtained from the active subspace method, which identifies low-dimensional structure in high-dimensional models by examining how the 
function’s gradients vary on average. The active subspace framework has been applied widely 
and provides an effective way to reduce dimensionality when accurate gradient information 
is available (Constantine \cite{constantine2015active}, Constantine et al. \cite{constantine2014active}).

The active subspace method and the resulting activity scores, however, have an important limitation: they depend heavily on accurate gradient 
information, which may be unavailable for models that are noisy, 
or non-smooth. These limitations create a need for sensitivity analysis 
tools that retain the interpretability and structural insights of active subspace while avoiding its 
dependence on derivatives.

In this paper, we introduce a new sensitivity measure, the \textit{global activity scores}, 
which eliminates the dependence on gradients. Global activity scores are based on 
identifying the important directions along which the function changes the most, where 
``change'' is quantified not through derivatives but through first-order finite differences 
of function values. Because finite differences capture variation over nonlocal regions of 
the input space, global activity scores provide a more \textit{global} assessment of sensitivity.
We investigate the theoretical relationship between global activity scores and Sobol' 
indices, and we present numerical results demonstrating the advantages of the method, 
particularly in the presence of noise. Our numerical results show that global activity scores are 
more robust than activity scores and derivative-based global sensitivity measures when the 
underlying function is noisy. This robustness is particularly desirable in applications where model evaluations are affected by noise or numerical variability, as it allows the method to identify variables that influence the underlying signal despite the presence of noise.

The remainder of the paper is organized as follows. Section \ref{sec:review} reviews Sobol' sensitivity indices (Sobol' \cite{sobol2001global}), derivative-based global sensitivity measures (Sobol' and Kucherenko \cite{sobol2009derivative}), and activity scores (Constantine and Diaz \cite{constantine2017global}). Section \ref{sec:GAscore} introduces our new sensitivity measure, \textit{global activity scores}. Section~\ref{sec:numerical} presents numerical experiments applying these methods to selected examples. We conclude in Section~\ref{sec:conc}.

\section{Global Sensitivity Measures}
\label{sec:review}
In this section we review some popular types of sensitivity measures, in particular, the variance-based and derivative-based measures. For a comprehensive survey of sensitivity measures, see Iooss and Lema{\^\i}tre \cite{iooss2015review}.

\subsection{Sobol' Sensitivity Indices}\label{sobol}
Consider a $d$-dimensional random input vector $\pmb x = (x_1, \ldots, x_d)$ following the uniform distribution on $(0,1)^d$, with a square-integrable real-valued function $f(\pmb x)$ defined on $(0,1)^d$, and let the index set be defined as $I = \{1,2,\ldots, d\}$. 
For $u\subseteq I$, the notation $\pmb x^{u}$ denotes a point in $(0,1)^{|u|}$, containing all components $x_j$ for $j\in u$.

The ANOVA decomposition of $f(\pmb x)$ is
$$
f(\pmb x) = \sum_{u \subseteq I}f_u(\pmb x),
$$
where the ``component" function $f_u$ is a function that only depends on $\pmb x^u$. The domain of $f_{u}$ is $(0,1)^{|u|}$, but it can be extended to $(0,1)^d$ by setting $f_u(\pmb x)=f_u(\pmb x^u)$. For the empty set, we put
$f_{\emptyset} = \int f(\pmb x) d\pmb x$. The component functions are constructed inductively - see Sobol' \cite{sobol2001global} for details. We have
% If we assume $\pmb x$ has uniform distribution on $(0,1)^d$, we can write
$$
\mathbb{E}[f(\pmb x)] = \int_{(0,1)^d} f(\pmb x) d\pmb x,
$$
and
$$
\text{Var}(f(\pmb x))=\sigma^2 = \int_{(0,1)^d} f^2(\pmb x) d\pmb x - \mathbb{E}[f(\pmb x)]^2.
$$
As a consequence of the orthogonality of the ANOVA decomposition, the variance of $f$ can be written as  
$$
\sigma^2 =\sum_{u \subseteq I}\sigma_u^2,
$$
where $\sigma_u^2$ is the variance of the component function $f_u$. If $u \neq \emptyset$, then
$$
\sigma_u^2 = \int_{(0,1)^d} f_u^2(\pmb x)d\pmb x - \left(\int_{(0,1)^d} f_u(\pmb x) d\pmb x\right)^2 = \int_{(0,1)^d} f_u^2(\pmb x)
d\pmb x.
$$
If $u = \emptyset$, then $\sigma_u^2=0$.
The Sobol' sensitivity indices for the subset $u$ are defined as
$$
\underline{S}_u = \frac{1}{\sigma^2}\sum_{v \subseteq u} \sigma_v^2 = \frac{\underline{\tau}_u^2}{\sigma^2}
\text{       and       }
\overline{S}_u = \frac{1}{\sigma^2}\sum_{v \bigcap u \neq \emptyset} \sigma_v^2 = \frac{\overline{\tau}_u^2}{\sigma^2}, 
$$ 
%where $\underline{S}_u $ is called the lower Sobol' sensitivity index (or, the closed effect) and $\overline{S}_u$ is called the upper Sobol' sensitivity index (or, the total effect).  If $\underline{S}_u$ is close to 1, then the parameters $\pmb x^u$ are viewed as important to the model output. In practice the case when $u$ is a singleton $\{i\}$ is often considered. If $\overline{S}_{\{i\}}$ is close to 0, we consider $x_i$ to be an unimportant variable, and freeze it to its mean value to simplify the model $f$.
where $\underline{S}_u$ is called the lower Sobol' sensitivity index (or the closed effect) and $\overline{S}_u$ is called the upper Sobol' sensitivity index (or the total effect). 
A large value of $\underline{S}_u$ indicates that the variables $\pmb x_u$ have an important contribution to the variance of the model output. In particular, if $\underline{S}_u$ is close to 1, then $\pmb x_u$ accounts for most of the output variance.
In practice, the case where $u$ is a singleton $\{i\}$ is often considered. If $\overline{S}_{\{i\}}$ is close to 0, then $x_i$ has little influence on the output, and it is common to fix $x_i$ at a nominal value (e.g., its mean) to simplify the model $f$.

Upper Sobol' indices can be written as (Sobol' \cite{sobol1993sensitivity})
\begin{equation}\label{eq:sobol}
\bar S_i=\frac1 {2\sigma^2}\int{(f(\pmb z)-f(\pmb v_{\{i\}}{:}\pmb z_{\{-i\}}))^2d\pmb zdv_{i}},
\end{equation}
where $\pmb z$ and $\pmb v$ are independent random vectors uniformly distributed on $(0,1)^d$, and $(\pmb v_{\{i\}}{:}\pmb z_{-\{i\}})$ means the vector whose $i$th component is $v_i$ and whose $j$th component $(j\neq i)$ is $z_j$. The notations $\pmb v_{\{i\}}$ and $v_i$ both mean the $i$th component of $\pmb v$.
Eqn. (\ref{eq:sobol}) can be generalized to random vectors $\pmb x = (x_1,\ldots, x_d)$ with distribution function (cumulative distribution function) $\pmb F(\pmb x) = F_1(x_1)\cdots F_d(x_d)$ (Kucherenko et al. \cite{kucherenko2012estimation}) as
\begin{equation}
\label{sobol_index_gen}
\bar S_i=\frac1 {2\sigma^2}\int{(f(\pmb z)-f(\pmb v_{\{i\}}{:}\pmb z_{\{-i\}}))^2d\pmb F(\pmb z)dF_i(v_{i})}.
\end{equation}
where $\pmb z$ and $\pmb v$ are independent random vectors with distribution $\pmb F$.

\subsection{Derivative-based sensitivity measures}

Computationally more efficient sensitivity measures can be obtained if the partial derivatives of $f$ exist.  For example, Campolongo et al. \cite{campolongo2007effective} introduced a global sensitivity measure based on $\int_{(0,1)^d}
\left | \frac{ \partial f(\pmb x)}{\partial x_i}\right | d\pmb x$, and Sobol' and Kucherenko \cite{sobol2009derivative} introduced a global sensitivity measure based on $\int_{(0,1)^d} (\frac{\partial f(\pmb x)}{\partial x_i})^2 d\pmb x$. Here we discuss the second measure further.

The derivative-based global sensitivity measure of Sobol' and Kucherenko \cite{sobol2009derivative} is given by
\begin{equation}
v_i = \int_{(0,1)^d}\left(\frac{\partial f(\pmb x)}{\partial x_i}\right)^2 d\pmb x = \mathbb{E}\left[\left(\frac{\partial f(\pmb x)}{\partial x_i}\right)^2\right].
\label{dgsm1}
\end{equation}
%DGSM can be estimated by Monte Carlo as
%\[
%v_i \approx \hat{v}_i = \frac{1}{N}\sum_{j=1}^N (\frac{\partial f(\pmb x^{(j)})}{\partial x_i})^2,
%\]
%where $\pmb x^{(j)},j=1,\ldots,N$ is a random sample from the uniform distribution on $(0,1)^d$.

Sobol' \cite{sobol2011derivative} showed that if $f(\pmb x)$ is a linear function on each of its components $x_i$, then
\begin{equation}
\overline{S}_i = \frac{1}{12}\frac{v_i}{\sigma^2}. 
\label{relationship1}
\end{equation}
In general, Sobol' and Kucherenko \cite{sobol2010derivative} established the following inequality
\begin{equation}
\label{ineq_upper}
\overline{S}_i \leq \frac{1}{\pi^2} \frac{v_i}{\sigma^2},
\end{equation}
where the constant $1/\pi^2$ arises from a Poincar\'e inequality for the uniform distribution.

The derivative-based global sensitivity measure (DGSM) can be generalized to nonuniform distributions, where the random vector $\pmb x=(x_1,\ldots x_d)$ has independent components with marginal distributions $F_1,\ldots,F_d$, and each $F_i$ has a corresponding density function $p_i$. Sobol' and Kucherenko \cite{sobol2010derivative} and Kucherenko and Iooss \cite{kucherenko2017derivative} obtained several results in this setting. In particular, the following generalization of inequality  (\ref{ineq_upper}) was obtained in 
\cite{lamboni2013derivative}
for distributions that satisfy certain regularity conditions:
\begin{equation}
\overline{S}_i \leq 4\left[\sup_{x\in \mathbb{R}} \frac{\min(F_i(x), 1-F_i(x))}{p_i(x)}\right]^2 \frac {v_i}{\sigma^2}.
\label{eq_dgsm}
\end{equation}
Improved versions of this inequality can be found in Roustant et al. \cite{roustant2017poincare}.

\subsection{Activity scores}

The active subspace (AS) method (Constantine et al. \cite{constantine2014active}) finds the important directions of a function using the gradient of the function, and uses the important directions to reduce the domain of the function to a subspace.  The activity score, introduced by Constantine and Diaz \cite{constantine2017global}, is a global sensitivity measure obtained from this information. 
%We will follow \cite{constantine2017global} to describe the activity score next. 

Consider a square-integrable function $f(\pmb x)$ defined on $(0,1)^d$, with finite partial derivatives that are square-integrable, and gradient
%Let $p(\pmb x)$ be a probability density function on $(0,1)^d$ and consider the gradient of $f$
$$\nabla f(\pmb x) = \left[\frac{\partial f(\pmb x)}{\partial x_1},\frac{\partial f(\pmb x)}{\partial x_2},\ldots, \frac{\partial f(\pmb x)}{\partial x_d}\right]^T. 
$$
Define the matrix $C_\text{as}$ as
\begin{equation}
C_\text{as} =  \mathbb{E}[\nabla f(\pmb x)\nabla f(\pmb x)^T],
\label{eigen_eqn}
\end{equation}
where the expectation is computed with respect to the uniform distribution. Let the eigenvalue decomposition of $C_\text{as}$ be
\[
C_\text{as}=\textbf{W}\Lambda_\text{as} \textbf{W}^T,
\]
where $\textbf{W} = [\textbf{w}_1,\ldots, \textbf{w}_d]$ is the $d\times d$ orthogonal matrix of eigenvectors, and $\Lambda_\text{as} = \mathrm{diag}(\lambda_1,\ldots, \lambda_d)$ with $\lambda_1 \geq \ldots \geq \lambda_d \geq 0$ is the diagonal matrix of eigenvalues in descending order.

Matrix $C_\text{as}$ can be approximated using the Monte Carlo method as
\begin{align}
\label{MC_as}
C_\text{as} \approx \hat{C}_\text{as} = \frac{1}{N}\sum_{j=1}^N(\nabla_{\pmb x} f(\pmb x^{({j})}))(\nabla_{\pmb x} (f(\pmb x^{(j)})))^T,
\end{align}
where the $\pmb x^{(1)},\ldots,\pmb x^{(N)}$ are a random sample of size $N$ from the density $p$.

If there exists an integer $m$ such that eigenvalues $\lambda_{m+1},\ldots,\lambda_d$ are sufficiently small, then the active subspace method approximates $f(\pmb x)$ with a lower dimensional function $g(\textbf{W}_1^T\textbf x)$ where $\textbf{W}_1$ is the $d\times m$ matrix containing the first $m$ eigenvectors. The dimension of $g$ is $m$, whereas the dimension of $f$ is $d$.

The activity score for the $i$th parameter is defined as
\begin{equation}
\alpha_i(m) = \sum_{j=1}^{m} \lambda_j w_{ij}^2,
\label{ac_1}
\end{equation}
where $\textbf{w}_j=[w_{1j}, \ldots, w_{dj}]^T$ is the $j$th eigenvector, $m\leq d$, and $i = 1,\ldots, d$.

Constantine and Diaz \cite{constantine2017global}  showed that the activity scores are bounded by DGSM 
$$
\alpha_i (m) \leq v_i,
$$
where $i=1,\ldots,d$, and the inequality becomes an equality if $m = d$.

The following inequality between the Sobol' sensitivity index and activity scores was also established in  \cite{constantine2017global}:
$$
\overline{S}_i\leq \frac{1}{\pi^2}\frac{\alpha_i(m) + \lambda_{m+1}}{\sigma^2},
$$
where $i=1,\ldots,d$, $m=1,\ldots,d-1$.
%The inequality is proved using 
%$$
%v_i = \alpha_i(m) + \sum_{j=m+1}^d \lambda_j w_{ij}^2 \leq \alpha_i(m) + \lambda_{m+1}\sum_{j=m+1}^d w_{ij}^2 \leq \alpha_i(m) + \lambda_{m+1}. 
%$$
Duan and \"{O}kten \cite{duan2023derivative} note that this result can be generalized to non-uniform distributions over $\mathbb{R}^d$ using Eqn. (\ref{eq_dgsm}) as
$$
\overline{S}_i\leq 4\left[\sup_{x\in \mathbb{R}} \frac{\min(F_i(x), 1-F_i(x))}{p_i(x)}\right]^2\frac{\alpha_i(m)+\lambda_{m+1}}{\sigma^2},
$$ 
where $F_i$ and $p_i$, $i=1,\ldots,d$, are marginal distributions and density functions for the components of the random input vector $\pmb x$.

\section{Global Activity Scores}
\label{sec:GAscore}

In this section, we introduce a new global sensitivity measure: global activity scores. This measure replaces the expectations of gradients used in traditional activity scores with expectations of finite differences. We will demonstrate, through numerical experiments on noisy problems, the advantages of using finite differences over gradients.

Consider a square-integrable real-valued bounded function $f(\pmb z)$ with domain $\Omega=\Omega_1\times\Omega_2\times\cdots\times\Omega_d\subset\mathbb R^d$ and continuous second-order partial derivatives. Suppose $\Omega$ is endowed with a probability measure with a cumulative distribution function in the form $ \pmb F(\pmb z)=F_1(z_1)\cdot \ldots \cdot F_d(z_d)$, where $F_i$ are marginal distribution functions. 

For vectors $\pmb v$ and $\pmb z$ in $\mathbb R^d$, let $(\pmb v_{\{i\}}{:}\pmb z_{-\{i\}})$ denote the vector formed by setting the $i$th component equal to $v_i$, and setting all other components $v_j$ ($j\neq i$) equal to $z_j$.
Here the notations $\pmb v_{\{i\}}$ and $v_i$ both mean the $i$th component of $\pmb v$. We use the former notation only when we are splicing and replacing a component from a vector with another.

We define the operator $D_{\pmb z,i}$ as follows:
\begin{equation}
\label{def_D}
D_{\pmb z,i}f(v_{i},\pmb z)=(f(\pmb v_{\{i\}}{:}\pmb z_{-\{i\}})-f(\pmb z))/(v_{i}-z_{i}).
\end{equation}
Then, the operator $D_{\pmb z}$, acting on $f$, has the value at $(\pmb v, \pmb z)$ given by 
\begin{equation}
\label{def_D0}
D_{\pmb z}f(\pmb v,\pmb z)=[D_{\pmb z,1}f(v_{1},\pmb z),\ldots,D_{\pmb z,d}f(v_{d},\pmb z)]^T.
\end{equation}

Define the $d\times d$ matrix $\pmb C_\text{gas}$ (here ``gas" means generalized active subspace) with the operator $D_{\pmb z}$ defined in Equation (\ref{def_D0}), as
\begin{equation}
\label{def_C}
\pmb C_\text{gas}=\E[\E[(D_{\pmb z}f)(D_{\pmb z}f)^T|\pmb z]], 
\end{equation}
where the inner conditional expectation fixes $\pmb z$ and averages over the components of $\pmb v=[v_1,\ldots,v_d]^T$, and the outside expectation averages with respect to $\pmb z$. Both $\pmb v$ and $\pmb z$ follow the same continuous probability distribution $\pmb F$, and they are independent.
The necessary conditions for the existence of the integrals in $\pmb C_\text{gas}$, for the case of bivariate $f$, are discussed in the Appendix - the proofs generalize to higher dimensions in a straightforward way.

Next we discuss the estimation of $\pmb C_\text{gas}$. Each entry of the matrix $\pmb C_\text{gas}$ is a double expectation. The inner expectations on the diagonal of $\pmb C_\text{gas}$ can be expressed as a one-dimensional integral, which we approximate using a deterministic quadrature rule. The outer expectation is a multidimensional integral, estimated by Monte Carlo. 
For example, consider a bivariate function and the first entry of the matrix:
\begin{align*}
C_{11}&=\int_{\mathbb{R}^3}\left(\frac{f(v_1,z_2)-f(z_1,z_2)}{v_1-z_1}\right)^2dF_1(v_1)d\pmb F(\pmb z).
\end{align*}
To estimate this integral, generate a sample $\pmb{z}^{(i)}=(z_1^{(i)},z_2^{(i)})$ from $\pmb F$, for $i=1,\ldots,N.$ The corresponding inner integral is
\begin{align}
\label{inner_one}
\int_{\mathbb{R}}\left(\frac{f(v_1,z^{(i)}_2)-f(z^{(i)}_1,z^{(i)}_2)}{v_1-z^{(i)}_1}\right)^2dF_1(v_1).
\end{align}
This integrand has a singularity at $v_1=z^{(i)}_1$. To handle it, split the domain at the singularity into 
$(-\infty,z^{(i)}_1)$ and $(z^{(i)}_1,\infty)$, and approximate each part using a deterministic quadrature rule such as Newton-Cotes or Gauss-Legendre. Repeating this procedure for $i=1,\ldots,N$ and averaging the results yields an estimate of $C_{11}$.

To describe the estimation of an off-diagonal element of the matrix, consider the integral
\begin{align*}
C_{12}&=\int_{\mathbb{R}^4}\left(\frac{f(v_1,z_2)-f(z_1,z_2)}{v_1-z_1}\right)\left(\frac{f(z_1,v_2)-f(z_1,z_2)}{v_2-z_2}\right) dF_1(v_1)dF_2(v_2)d\pmb F(\pmb z).
\end{align*}
To estimate this integral, generate a sample $\pmb{z}^{(i)}=(z_1^{(i)},z_2^{(i)})$ from $\pmb F$, for $i=1,\ldots,N.$ The corresponding inner integral is
\begin{align}
\label{inner_two}
&\int_{\mathbb{R}^2}\left(\frac{f(v_1,z^{(i)}_2)-f(z^{(i)}_1,z^{(i)}_2)}{v_1-z^{(i)}_1}\right)\left(\frac{f(z^{(i)}_1,v_2)-f(z^{(i)}_1,z^{(i)}_2)}{v_2-z^{(i)}_2}\right)dF_1(v_1)dF_2(v_2)\\
&=\int_{\mathbb{R}}\left(\frac{f(v_1,z^{(i)}_2)-f(z^{(i)}_1,z^{(i)}_2)}{v_1-z^{(i)}_1}\right)dF_1(v_1)\int_{\mathbb{R}}\left(\frac{f(z^{(i)}_1,v_2)-f(z^{(i)}_1,z^{(i)}_2)}{v_2-z^{(i)}_2}\right)dF_2(v_2).
\end{align}
Then, compute each one-dimensional integral separately using the same methodology as for the inner integral of $C_{11}$.

Consider the eigenvalue decomposition of $\pmb C_\text{gas}$
\begin{equation}
\pmb C_\text{gas}=\pmb U\Lambda_\text{gas} \pmb U^T,
\end{equation}
where $\pmb U = [\textbf{u}_1,\ldots, \textbf{u}_d]$ is the $d\times d$ orthogonal matrix of eigenvectors, and $\Lambda_\text{gas} = \mathrm{diag}(\lambda_1,\ldots, \lambda_d)$ with $\lambda_1 \geq \ldots \geq \lambda_d \geq 0$ is the diagonal matrix of eigenvalues in descending order. 
The global activity scores are defined based on the eigenvalue decomposition of $\pmb C_\text{gas}$.

\begin{definition}
The global activity score for the $i$th parameter, $1\leq i \leq d$, is
\begin{equation}\label{equ:gas}
\gamma_i(m)=\sum_{j=1}^{m}{\lambda_ju_{ij}^2},
\end{equation}
where $\lambda_1,\ldots,\lambda_d$ are the eigenvalues from $\Lambda_\text{gas}$, $\pmb u_j = [u_{1j},\ldots,u_{dj}]^T$ is the $j$th eigenvector from $\pmb U$, and $m \leq d$.
\end{definition}

The next theorem presents an inequality between the upper Sobol' sensitivity index $\bar S_i$ and $\gamma_i(d)$.

\begin{theorem}\label{gas1}
If $\Omega =(0,1)^d$, endowed with the uniform probability measure, then
\begin{equation}
\bar S_i\leq\frac12\frac{\gamma_i(d)}{\sigma^2},
\end{equation}
for $i=1,\ldots,d$.
\end{theorem}

\begin{proof}
Consider the upper Sobol' index $\bar S_i$ (Eqn. (\ref{eq:sobol}))
\[
\bar S_i=\frac1 {2\sigma^2}\int{(f(\pmb z)-f(\pmb v_{\{i\}}{:}\pmb z_{\{-i\}}))^2d\pmb zdv_{i}}.
\]
Since $(v_i-z_i)^2\leq 1$ for any $v_i, z_i$ in $(0,1)$, we have
\[
\bar S_i\leq\frac{1}{2\sigma^2}\int{((f(\pmb z)-f(\pmb v_{\{i\}}{:}\pmb z_{\{-i\}}))/(z_{i}-v_{i}))^2}d\pmb zdv_{i}.
\]
Let $\pmb{\bar S}=[\bar S_1,\ldots,\bar S_d]^T$. We have the componentwise inequality
\[
\pmb{\bar S}\leq \frac{1}{2\sigma^2}\mathrm{diag}(\pmb C_\text{gas})=\frac{1}{2\sigma^2}\mathrm{diag}(\pmb U\Lambda_{\text{gas}} \pmb U^T)\]
where $\mathrm{diag}(A)$ is the vector obtained from the diagonal elements of the matrix $A$. Since the $i$th value of $\mathrm{diag}(\pmb U\Lambda_{\text{gas}} \pmb U^T)$ is $\gamma_i(d)$ by definition of $\gamma_i(d)$, the above equality implies
\[
\bar S_i \leq \frac{1}{2\sigma^2}\gamma_i(d)
\]
for $i=1,\ldots,d.$
\end{proof}

The following result links any $\gamma_i(m),1\leq m \leq d-1$, to the upper Sobol' index.

\begin{corollary}\label{gas2}
If $\Omega =(0,1)^d$, endowed with the uniform probability measure, then
\begin{equation}
\bar S_i\leq\frac12\frac{\gamma_i(m)+\lambda_{m+1}}{\sigma^2},
\end{equation}
where $i=1,\ldots,d$, and $\lambda_{m+1}$ is the $(m+1)$th eigenvalue of matrix $\pmb C_\text{gas}$.  
\end{corollary}
\begin{proof}
The proof follows from Theorem \ref{gas1}, the orthogonality of $\pmb U$, and the following inequality:
\[
\gamma_i(d)=\gamma_i(m)+\sum_{j=m+1}^{d}{\lambda_ju_{ij}^2}\leq\gamma_i(m)+\lambda_{m+1}\sum_{j=m+1}^{d}{u_{ij}^2}.
\]
\end{proof}

Theorem \ref{gas1} and Corollary \ref{gas2} can be generalized to unbounded domains and nonuniform measures. Let $\Omega = \mathbb{R}^d$, endowed with a probability distribution function in the form $\pmb F(\pmb z)=F_1(z_1)\cdot \ldots \cdot F_d(z_d)$.

\begin{theorem}\label{gas_general}
Let $f(\pmb z)$ be a bounded function on $\Omega$. For any $0<\epsilon<1$
\begin{align}
\bar S_i \leq \frac{(b'-a')^2}{2} \frac{\gamma_i(d)+\kappa}{\sigma^2},
\end{align}
where $a',b' \in \mathbb{R}$ are such that $\int_{(a',b')^d} d\pmb F(\pmb z)=1-\epsilon$, $i=1,\ldots,d$, and $\kappa$ is a positive constant given by
\begin{align}
\kappa=\frac{2\epsilon-\epsilon^2}{(b'-a')^2}\sup_{\pmb z,\pmb v\in\Omega}(f(\pmb z)-f(\pmb v))^2.   
\end{align}
\end{theorem}

\begin{proof}
From the definition of generalized upper Sobol' indices (Eqn. (\ref{sobol_index_gen})), we have
\begin{align*}
\bar S_i&=\frac1 {2\sigma^2}\int{(f(\pmb z)-f(\pmb v_{\{i\}}{:}\pmb z_{\{-i\}}))^2d\pmb F(\pmb z)dF_i(v_{i})}\\
&=\frac1 {2\sigma^2}\int_{\Omega\times\Omega}{(f(\pmb z)-f(\pmb v_{\{i\}}{:}\pmb z_{\{-i\}}))^2d\pmb F(\pmb z)d\pmb F(\pmb v)}.
\end{align*}
Let $\Omega'=(a',b')^d$, and write $\Omega\times\Omega$ as the union of $\Omega'\times\Omega'$ and the complement $(\Omega'\times\Omega')^c$. Since $(z_i-v_i)^2 \leq (b'-a')^2$, for any $z_i,v_i \in (a',b')$, we have
\begin{align*}
&\int_{\Omega'\times\Omega'}{(f(\pmb z)-f(\pmb v_{\{i\}}{:}\pmb z_{\{-i\}}))^2 d\pmb F(\pmb z)d\pmb F(\pmb v)}\\
&\leq (b'-a')^2 \int_{\Omega'\times\Omega'}{((f(\pmb z)-f(\pmb v_{\{i\}}{:}\pmb z_{-\{i\}}))/(z_{i}-v_{i}))^2d\pmb F(\pmb z)d\pmb F(\pmb v)}.
\end{align*}
The integral over the complement of $\Omega'\times\Omega'$ can be bounded by
\begin{align*}
&\int_{(\Omega'\times\Omega')^c}{(f(\pmb z)-f(\pmb v_{\{i\}}{:}\pmb z_{\{-i\}}))^2d\pmb F(\pmb z)d\pmb F(\pmb v)}
\leq (b'-a')^2 \kappa.
\end{align*}

Using the bounds for the two integrals and dividing by $2\sigma^2$ we obtain
\begin{align}
\label{proof_upper}
\bar S_i &\leq \frac{(b'-a')^2}{2 \sigma^2}\int_{\Omega \times \Omega}{((f(\pmb z)-f(\pmb v_{\{i\}}{:}\pmb z_{-\{i\}}))/(z_{i}-v_{i}))^2d\pmb F(\pmb z)d\pmb F(\pmb v)}+ \frac{(b'-a')^2}{2 \sigma^2} \kappa \\
&= \frac{(b'-a')^2}{2 \sigma^2}\int_{\Omega \times \mathbb{R}}{((f(\pmb z)-f(\pmb v_{\{i\}}{:}\pmb z_{-\{i\}}))/(z_{i}-v_{i}))^2d\pmb F(\pmb z)dF_i(v_{i})}+ \frac{(b'-a')^2}{2 \sigma^2} \kappa. \nonumber
\end{align}
Let $\pmb{\bar S}=[\bar S_1,\ldots,\bar S_d]^T$. Inequality (\ref{proof_upper}) can be written as
\begin{align*}
\pmb{\bar S}&\leq \frac{(b'-a')^2}{2\sigma^2} \left(\mathrm{diag}(\pmb C_\text{gas})+\kappa\pmb{1_d}\right)
=\frac{(b'-a')^2}{2\sigma^2} \left(\mathrm{diag}(\pmb U\Lambda_\text{gas} \pmb U^T)+\kappa\pmb{1_d}\right),
\end{align*}
where $\pmb{1_d}$ denotes the $d$-dimensional vector of ones. Since $\gamma_i(d)$ is the $i$th value of $\mathrm{diag}(\pmb U\Lambda_{\text{gas}} \pmb U^T)$, the proof is over.
\end{proof}

\begin{remark}
If the domain of $f$ is a finite rectangle $(a,b)^d$, then the integral over the unbounded domain in the proof of Theorem \ref{gas_general} vanishes, and we obtain the bound
\begin{align*}
\bar S_i \leq \frac{(b-a)^2}{2} \frac{\gamma_i(d)}{\sigma^2}.
\end{align*}
If we put $a=0,b=1$ in the above inequality, we obtain the bound in Theorem \ref{gas1}.
\end{remark}

For any $m$, $1\leq m \leq d-1$, we have the following inequality for $\gamma_i(m)$ and $\bar S_i$.

\begin{corollary}
\label{cor}
With the same assumptions as in Theorem \ref{gas_general}, we have
\begin{equation}
\bar S_i\leq \frac{(b'-a')^2}{2}\frac{\gamma_i(m)+\lambda_{m+1}+\kappa}{\sigma^2},
\end{equation}
where $i=1,\ldots,d$ and $\lambda_{m+1}$ is the $(m+1)$th eigenvalue of matrix $\pmb C_\text{gas}$.  
\end{corollary}
\begin{proof}
The proof follows from Theorem \ref{gas_general}, as in the proof of Corollary \ref{gas2}.
\end{proof}

\begin{remark}
If the domain of $f$ is a finite rectangle $(a,b)^d$, then the bound in Corollary \ref{cor} simplifies as
\begin{align*}
\bar S_i\leq \frac{(b-a)^2}{2}\frac{\gamma_i(m)+\lambda_{m+1}}{\sigma^2}.
\end{align*}
If we put $a=0,b=1$ in the above inequality, we obtain the bound in Corollary \ref{gas2}.
\end{remark}

The next theorem presents a much simpler relationship between upper Sobol' indices and global activity scores for quadratic functions.

\begin{theorem}\label{gas3}
If $\Omega =\mathbb R^d$ and $\pmb{z}$ has the multivariate standard normal probability distribution $\pmb F = (F_1,\ldots,F_d)$, and $f(\pmb z)=\frac12\pmb z^T\pmb{Az}+\pmb{b}^T \pmb{z}$, where $\pmb A$ is a symmetric $d\times d$ matrix, then
\begin{equation}
\bar S_i=\frac{\gamma_i(d)}{\sigma^2},
\end{equation}
for $i=1,\ldots,d$.
\end{theorem}

\begin{proof}
We have $\pmb z=(z_1,\ldots,z_d)^T$, and $\pmb b = (b_1,\ldots,b_d)$. Define
\[
\pmb x_i=(z_1,\ldots,z_{i-1},\frac{z_i+v_i}{2},z_{i+1},\ldots,z_d)^T.
\]
Then
\begin{align*}
2\sigma^2\bar S_i&=\int{(f(\pmb z)-f(\pmb v_{\{i\}}{:}\pmb z_{\{-i\}}))^2d\pmb F(\pmb z)dF_i(v_{i})}\\
&=\int{\left(\frac{\partial f(\pmb x_i)}{\partial z_i}\right)^2(z_i-v_i)^2d\pmb F(\pmb z)dF_i(v_{i})}\\
&=\E\left[\left(\frac{\partial f(\pmb x_i)}{\partial z_i}\right)^2(z_i-v_i)^2\right].
\end{align*}
%where $\pmb x_i=(z_1,\ldots,z_{i-1},\frac{z_i+v_i}{2},z_{i+1},\ldots,z_d)^T$. 
%and the second equality is verified by substituting $f$ into both sides of the equality. 
Using the independence of $z_i-v_i$, $z_i+v_i$, and $z_j$, with $j\neq i$, we write the above expectation as
\begin{equation*}
2\sigma^2\bar S_i=\E\left[\left(\frac{\partial f(\pmb x_i)}{\partial z_i}\right)^2\right]\times \E[(z_i-v_i)^2]=2\E\left[\left(\frac{\partial f(\pmb x_i)}{\partial z_i}\right)^2\right].
\end{equation*}
Finally, we observe,
\begin{align*}
\gamma_i(d)&=\int{((f(\pmb z)-f(\pmb v_{\{i\}}{:}\pmb z_{\{-i\}}))/(z_{i}-v_{i}))^2d\pmb F(\pmb z)dF_i(v_{i})}\\
&=\int{\left(\frac{\partial f(\pmb x_i)}{\partial z_i}\right)^2d\pmb F(\pmb z)dF_i(v_{i})}\\
&=\sigma^2 \bar{S}_i.
\end{align*}
\end{proof}

%Note that we can generalize the proof above, and get exactly the same conclusion again if defining sensitivity criterion $\bar S_{\theta}$ and $\gamma_{\theta}(d)$ along an arbitrary direction $\theta$. The discussions on how to introduce the generalized criterion can be found in Liu and Owen \cite{liu2023preintegration}. 

\begin{remark}
For the activity scores, the relationship between the upper Sobol' index and the activity score for quadratic functions is an inequality. Indeed, under the same hypothesis of Theorem \ref{gas3}, Liu and Owen \cite{liu2023preintegration} shows that
\[
\bar S_i\leq \frac{\alpha_i(d)}{\sigma^2},
\]
for $i=1,\ldots,d$.
\end{remark}
%The theorem above showcases that we have equality relationship between $\bar S_{i}$ and $\gamma_{i}(d)$ for quadratic function with normal inputs. This conclusion is stronger than the similar relationship investigated between Sobol' criterion and active subspace criterion, which is not a strict equality (see Liu and Owen \cite{liu2023preintegration}). For linear functions with any distribution functions, there exist equality relationships between $v_i$ and $\bar S_{i}$ (for example, Eqn. (\ref{relationship1})), and between $\gamma_{i}(d)$ and $\bar S_{i}$ as global activity scores equal to activity scores in linear case. This implies that, in some way, global activity scores approximate upper Sobol' sensitivity indices better than activity scores / DGSMs.

\section{Numerical Results}
\label{sec:numerical}
In this section we compare four sensitivity measures; Sobol' indices, DGSMs, activity scores, and global activity scores, numerically, when they are used to conduct the global sensitivity analysis in some examples. 
%In addition, we compare the measures with $\mu^*$ from Morris method in the first and third example.

We use the algorithm in Sobol' \cite{sobol2001global} to estimate the upper Sobol' indices, and the Correlation-2 algorithm in Owen \cite{owen2013better} for the lower Sobol' indices. To reduce computational cost, we use the same sample to estimate all indices rather than generating an independent sample for each one. We use this approach for all the sensitivity measures. With this approach, when $N$ Monte Carlo samples are generated, the number of function evaluations is $(d+1) N$ for upper Sobol', and $(2d+2) N$ for lower Sobol' indices. To compute the DGSM (Eqn. (\ref{dgsm1})) using Monte Carlo, we generate $N$ samples, and for each sample we estimate derivatives using forward finite differences with an increment of $h=0.001$. The total number of function evaluations is $(d+1) N$.

To compute the activity scores $\alpha_i(m)$ (Eqn. (\ref{ac_1})), we use the approach of Constantine et al. \cite{constantine2014active}. We compute the SVD of the $d\times N$ matrix $[\nabla f(\pmb x^{(1)})\cdots \nabla f(\pmb x^{(N)})]$, where $x^{(i)},i=1,\ldots,N$, are the Monte Carlo samples, rather than computing the eigenvalue decomposition of $\pmb C_{\text{as}}$. We use forward finite differences with an increment of $h=0.001$ to estimate the derivatives. The total number of function evaluations is $(d+1) N$. 

To compute the global activity scores $\gamma_i(m)$ (Eqn.~(\ref{equ:gas})), we directly estimate $\pmb C_{\text{gas}}$ and compute its eigenvalue decomposition. To estimate the integrals in $\pmb C_{\text{gas}}$, we split the one-dimensional inner integrals at their singularities (see (\ref{inner_one}, \ref{inner_two})), and then estimate each of the resulting integrals using a 5-point Gauss-Legendre rule. The outer integral is estimated using a Monte Carlo sample of size $N$. The total number of function evaluations in estimating $\pmb C_{\text{gas}}$ is $(10d+1) N$ (note that the estimation of the integrals share many common function evaluations).
Computer codes for some of the examples are available at: \url{https://github.com/RuilongYue/global-activity-scores}.

An important step in the calculation of activity and global activity scores is determining the dimension $m$ of the active subspace and global active subspace that approximates the original input domain. Consider the eigenvalues $\lambda_1 \geq \ldots \geq \lambda_d \geq 0$ of $\Lambda_\text{as}$ and $\Lambda_\text{gas}$. Define the $m$th normalized cumulative sum of eigenvalues as $\sum_{k=1}^m\hat\lambda_k/\sum_{k=1}^d\hat\lambda_k$, where $\hat\lambda_k$ is the $k$th estimated eigenvalue. In the following numerical results, for each method, we will pick $m$ as the smallest integer for which the cumulative normalized sum of eigenvalues for the corresponding method is greater than $90\%$.

Appendix \ref{appendixB} summarizes the number of function evaluations and computational time for all the examples considered in this section. 

\subsection{Example 1: A test function with noise}
Consider the following function
\begin{equation}
f(\pmb z)=\sum_{i=1}^{10}iz_i+10(z_1z_2 - z_9z_{10})+k\epsilon,\epsilon\sim N(0,1),
\end{equation}
where $\pmb z = (z_1,z_2,\ldots,z_{10})$. The function has a linear component with inputs in increasing order of importance, and a component with second order interactions. The inputs are independent and have the uniform distribution on $(-0.5,0.5)$. We assume the function is evaluated with some noise modeled by the term $k\epsilon$, where $k$ is a positive constant and $\epsilon$ is a standard normal random variable. Akin to the global sensitivity analysis of stochastic codes (see Fort et al. \cite{fort2021global}, Hart et al. \cite{hart2017efficient}, and Nanty et al. \cite{nanty2016sampling}), we want to investigate the accuracy of various global sensitivity measures as the level of noise, $k$, increases. In the numerical results that follow, we consider four cases: no noise ($k=0$), low noise ($k=0.01$), moderate noise ($k=0.1$), and high noise ($k=1$).

In the numerical results, we use a Monte Carlo sample size of $N=20{,}000$ for the Sobol' indices, DGSM, and activity scores. For the global activity scores, we use a Monte Carlo sample size of $N=2{,}000$. 

\subsubsection{No noise ($k=0$)}

Fig. \ref{fig:1O} plots the normalized cumulative sum of eigenvalues of $\pmb{C}_{\text{as}}$ and $\pmb{C}_{\text{gas}}$, and the components of the eigenvector that corresponds to the largest eigenvalue for each method. The methods give virtually identical results.

\begin{figure}[h] % Defines figure environment
\centering
\subfloat{\includegraphics[width=0.5\textwidth]{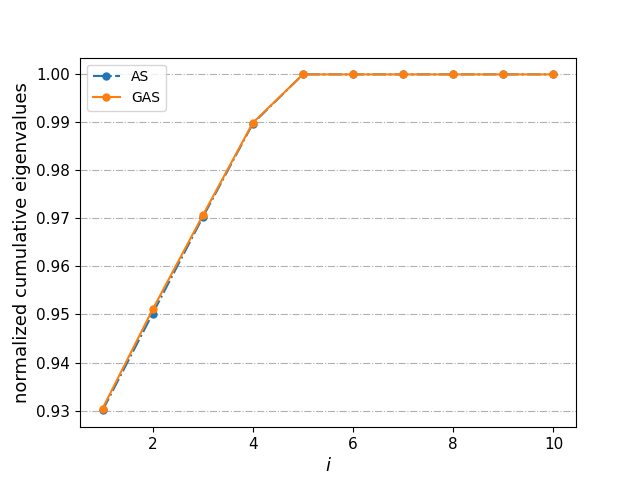}}
\subfloat{\includegraphics[width=0.5\textwidth]{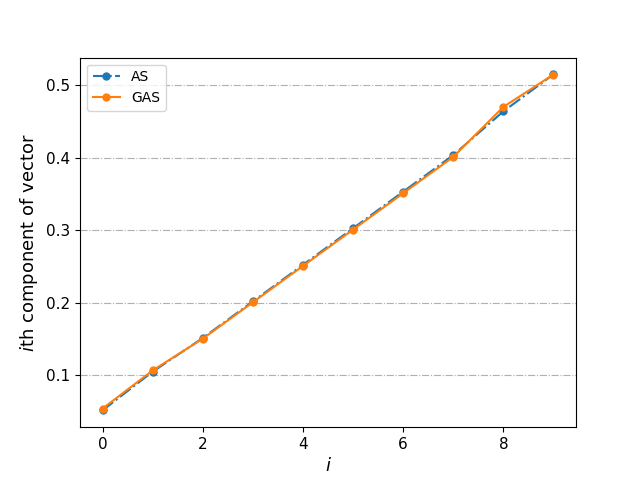}}\\
\caption{Normalized cumulative sum of eigenvalues (left) and the first eigenvector (right) of $\pmb C_\text{as}$ and $\pmb C_\text{gas}$, when $k=0$, in Example 1}
\label{fig:1O}
\end{figure}

Fig. \ref{fig:itr_noise_1} plots Sobol' lower and upper indices, normalized DGSMs, normalized activity scores, and normalized global activity scores.
%and normalized $\mu^*$ from Morris method. 
Normalization is done by dividing each score by the sum of the corresponding scores for DGSMs, activity scores, and global activity scores. We do not normalize Sobol' indices, for we want to see the unscaled differences between lower and upper Sobol' indices. For the global activity scores plot, we include the results for $m=1$ and $m=10$. For the activity scores, we only show the results for $m=1$, since $m=10$ for activity scores corresponds to DGSM. 

The upper Sobol' indices and DGSM follow the same pattern. When $m=1$, the activity scores and global activity scores are virtually identical, and they devalue the importance of the first two parameters compared to Sobol' indices and DGSMs, leading to a discrepancy in the ranking of the least three important inputs. When $m=10$, however, all the methods give the same ranking for the importance of inputs.
%, including Morris method.

\begin{figure}[h] % Defines figure environment
\centering
\subfloat[Sobol' indices]{\includegraphics[width=0.5\textwidth]{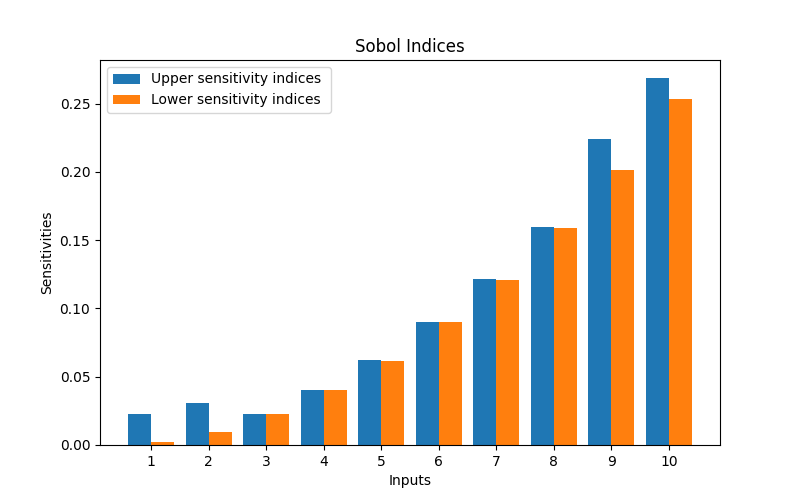}}
\subfloat[Normalized DGSMs]{\includegraphics[width=0.5\textwidth]{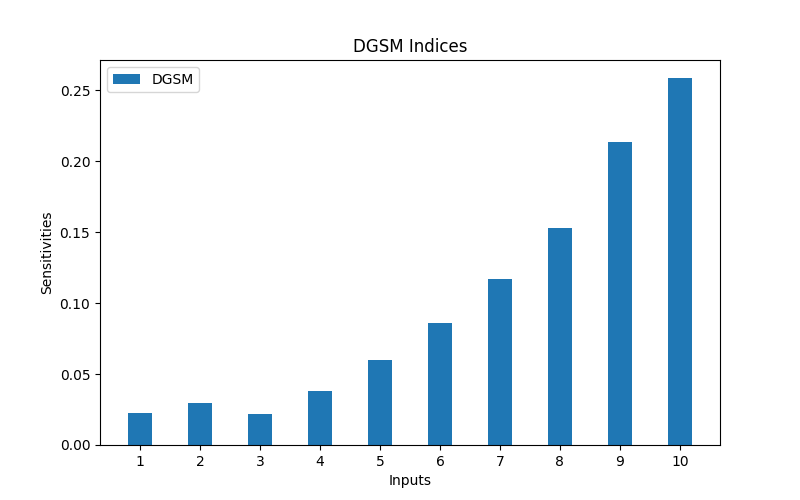}}\\
\subfloat[Normalized activity scores]{\includegraphics[width=0.5\textwidth]{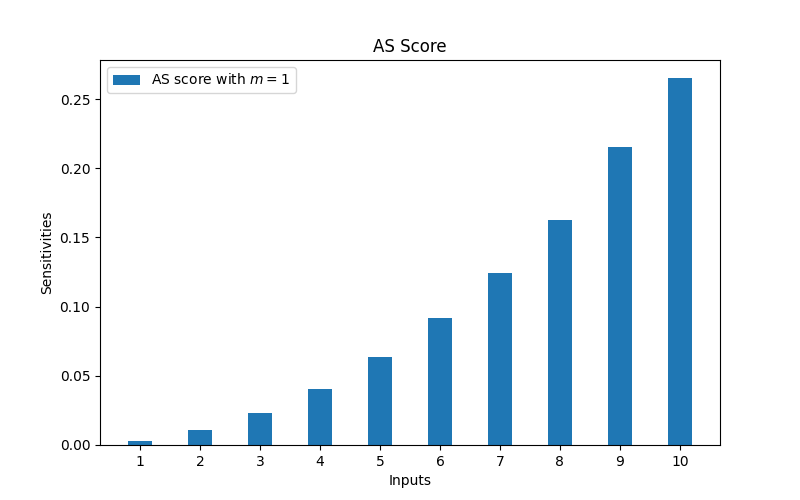}}
\subfloat[Normalized global activity scores]{\includegraphics[width=0.5\textwidth]{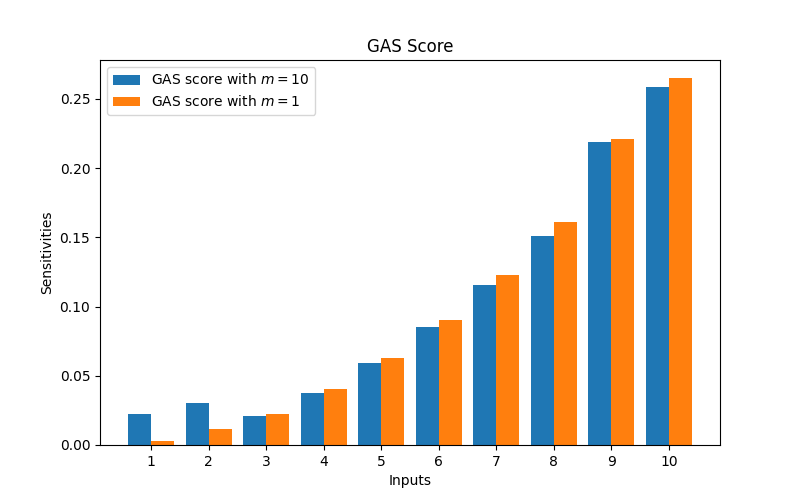}}\\
%\subfloat[Normalized $\mu^*$ of Morris method]{\includegraphics[width=0.5\textwidth]{figs/morris10.png}}\\
\caption{Sensitivity indices when $k=0$, in Example 1}
\label{fig:itr_noise_1}
\end{figure}

\subsubsection{Low noise ($k=0.01$)}

Fig. \ref{fig:O01} plots the normalized cumulative sum of eigenvalues of $\pmb{C}_{\text{as}}$ and $\pmb{C}_{\text{gas}}$ in the presence of low noise. We observe the methods have significantly different eigenvalues.

\begin{figure}[h] % Defines figure environment
\centering
\subfloat{\includegraphics[width=0.5\textwidth]{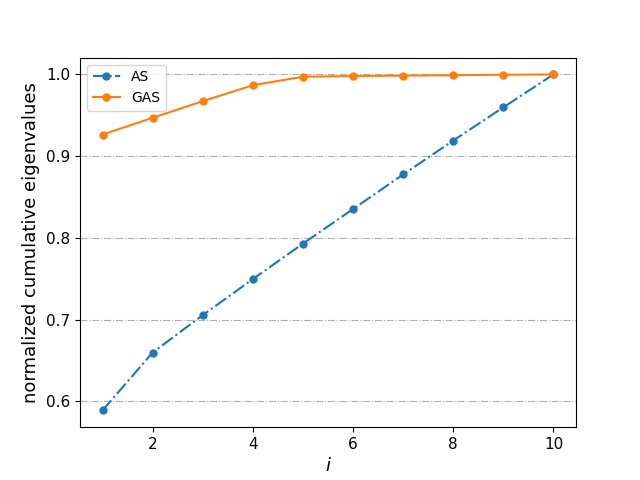}}
\subfloat{\includegraphics[width=0.5\textwidth]{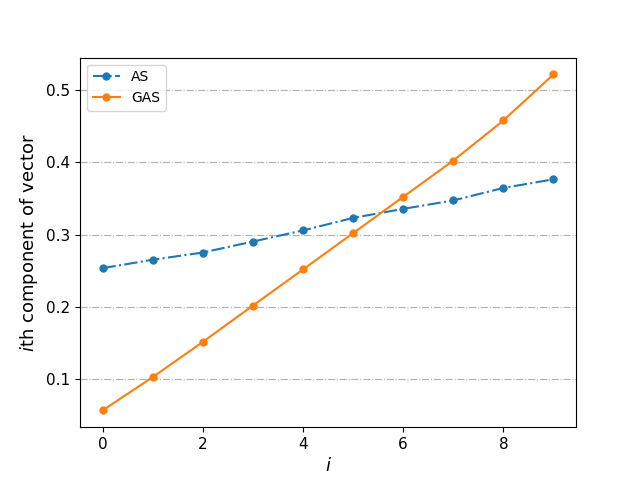}}\\
\caption{Normalized cumulative sum of eigenvalues (left) and the first eigenvector (right) of $\pmb C_\text{as}$ and $\pmb C_\text{gas}$, when $k=0.01$, in Example 1}
\label{fig:O01}
\end{figure}

Fig. \ref{fig:itr_noise_2} plots the sensitivity indices. Based on the cumulative normalized eigenvalue criterion discussed earlier, we pick $m=8$ for the activity scores, and $m=1$ for the global activity scores. With the introduction of noise, we observe that DGSM and activity scores do not look similar to Sobol' indices (or global activity scores) anymore. Even though DGSM still ranks most of the inputs as the upper Sobol' indices, the relative differences between the importance of the inputs are very different than that of Sobol' indices. The activity scores with $m=8$ rank several inputs differently than the Sobol' upper indices.

The global activity scores, on the other hand, do not seem to be adversely affected by noise.
%similarly like $\mu^*$ of Morris method. 
When $m=10$, the global activity scores and the upper Sobol' indices are virtually indistinguishable. When $m=1$, the global activity scores rank the first three inputs (which are the least important) differently compared to upper Sobol' indices, and interestingly follow the same pattern as lower Sobol' indices.

\begin{figure}[h] % Defines figure environment
\centering
\subfloat[Sobol' indices]{\includegraphics[width=0.5\textwidth]{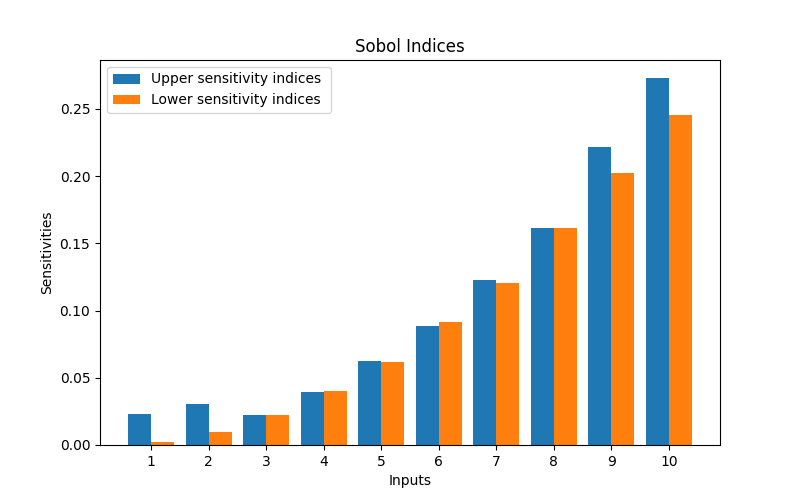}}
\subfloat[Normalized DGSMs]{\includegraphics[width=0.5\textwidth]{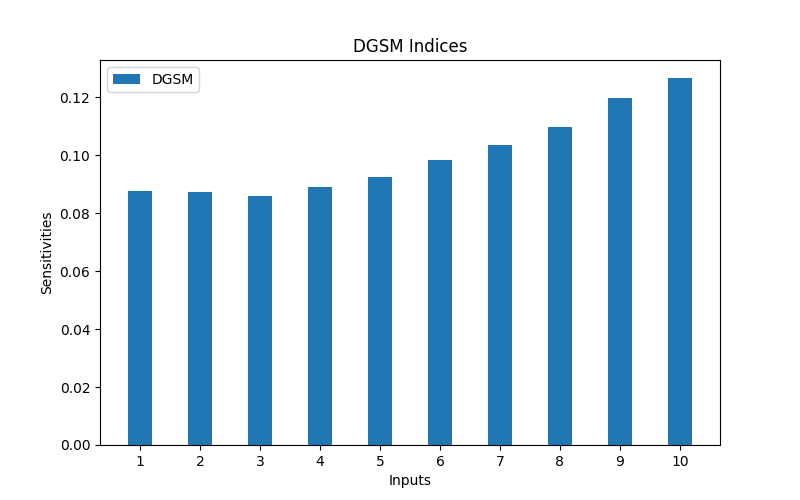}}\\
\subfloat[Normalized activity scores]{\includegraphics[width=0.5\textwidth]{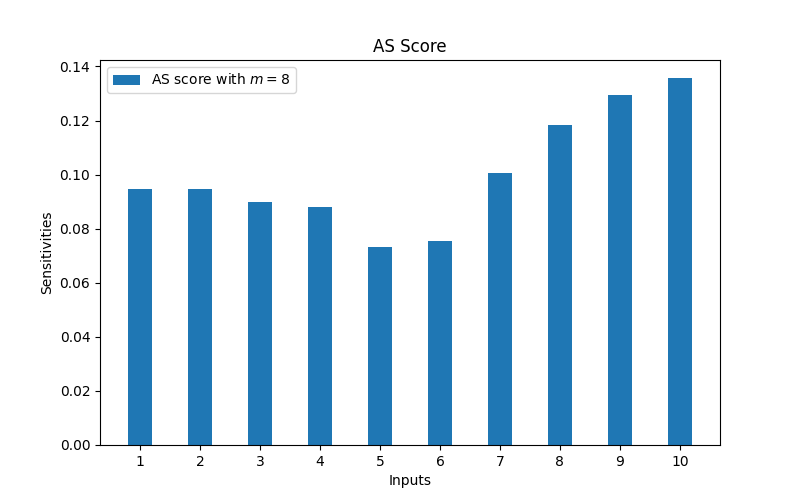}}
\subfloat[Normalized global activity scores]{\includegraphics[width=0.5\textwidth]{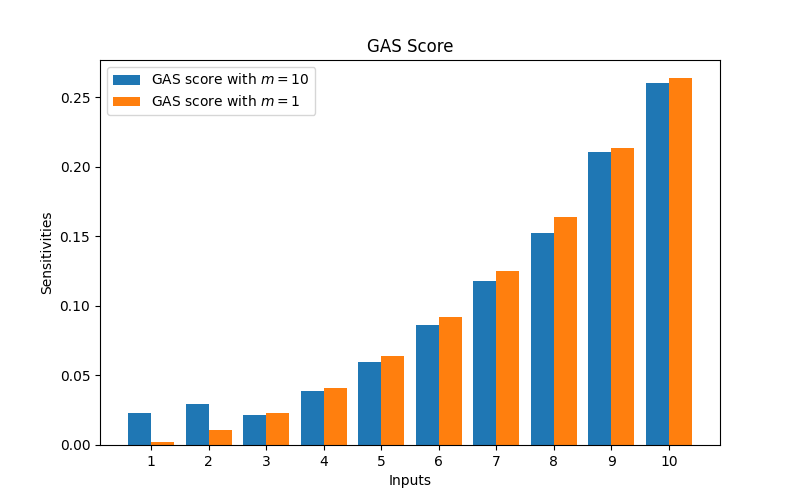}}\\
%\subfloat[Normalized $\mu^*$ of Morris method]{\includegraphics[width=0.5\textwidth]{figs/morris001.png}}\\
\caption{Sensitivity indices when $k=0.01$, in Example 1}
\label{fig:itr_noise_2}
\end{figure}

\subsubsection{Moderate noise ($k=0.1$)}
Fig. \ref{fig:01} plots the eigenvalues and the first eigenvector, and Fig. \ref{fig:itr_noise_3_5} plots the sensitivity indices for the case of high noise. Based on the eigenvalues, we pick $m=8$ for activity scores, and $m=7$ for global activity scores. DGSM and activity scores fail in identifying the important variables. Global activity scores with $m=10$ ranks input parameters the same way as upper Sobol' indices, except for the ranking of inputs 1 and 6.

\begin{figure}[h] % Defines figure environment
\centering
\subfloat{\includegraphics[width=0.5\textwidth]{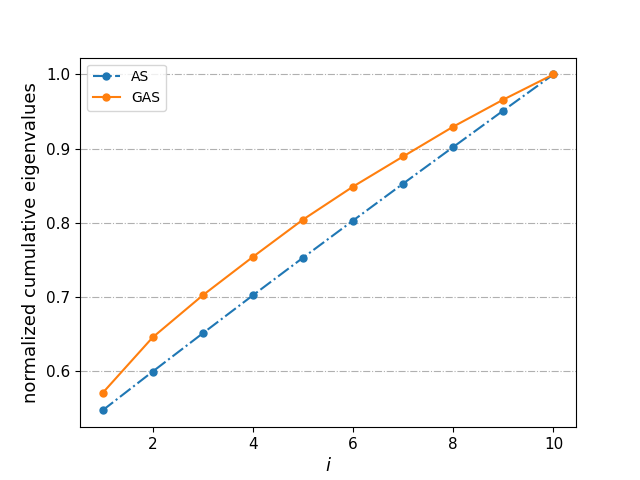}}
\subfloat{\includegraphics[width=0.5\textwidth]{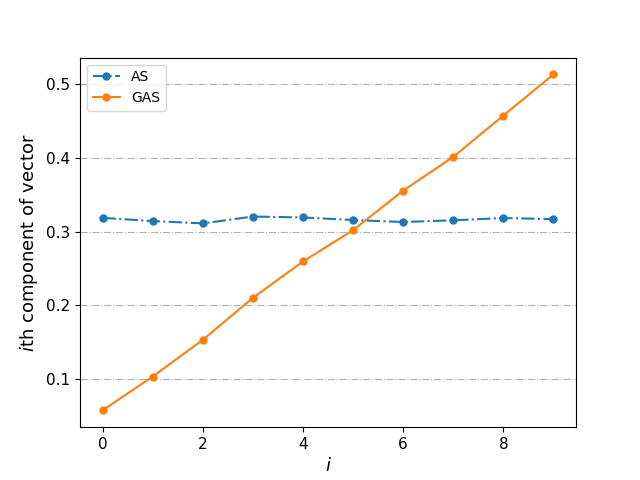}}\\
\caption{Normalized cumulative sum of eigenvalues (left) and the first eigenvector (right) of $\pmb C_\text{as}$ and $\pmb C_\text{gas}$, when $k=0.1$, in Example 1}
\label{fig:01}
\end{figure}

\begin{figure}[h] % Defines figure environment
\centering
\subfloat[Sobol' indices]{\includegraphics[width=0.5\textwidth]{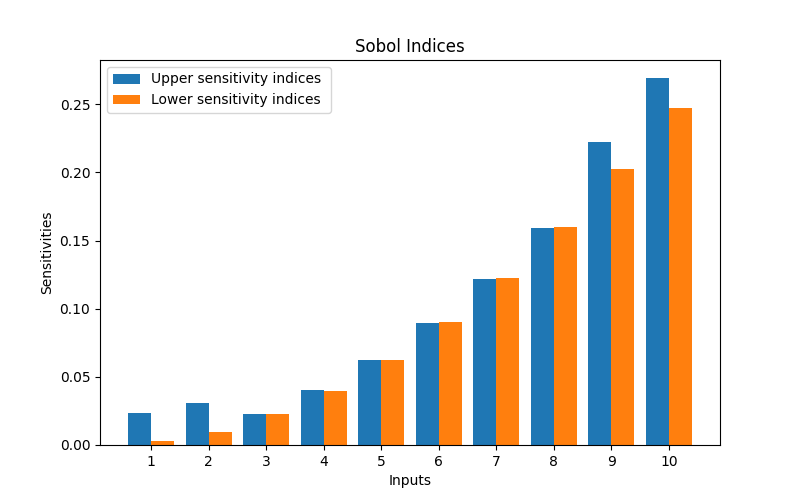}}
\subfloat[Normalized DGSMs]{\includegraphics[width=0.5\textwidth]{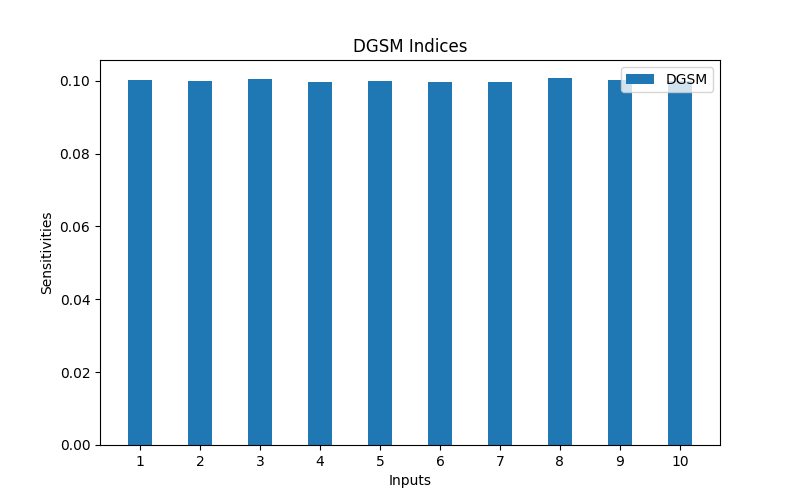}}\\
\subfloat[Normalized activity scores]{\includegraphics[width=0.5\textwidth]{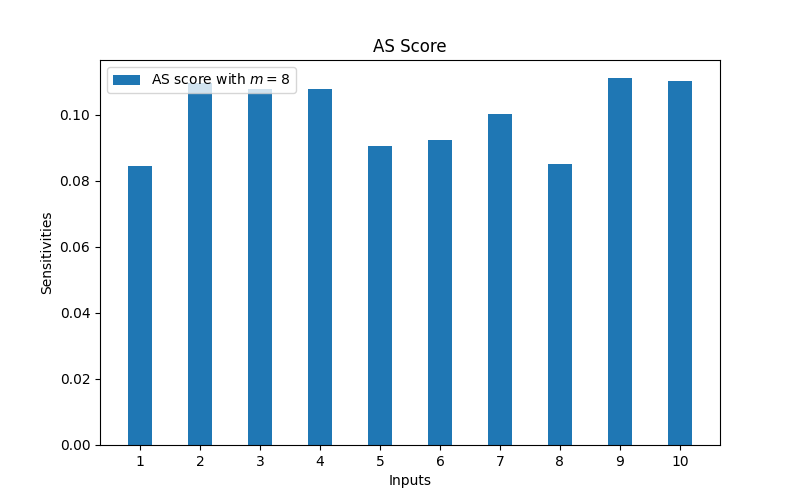}}
\subfloat[Normalized global activity scores]{\includegraphics[width=0.5\textwidth]{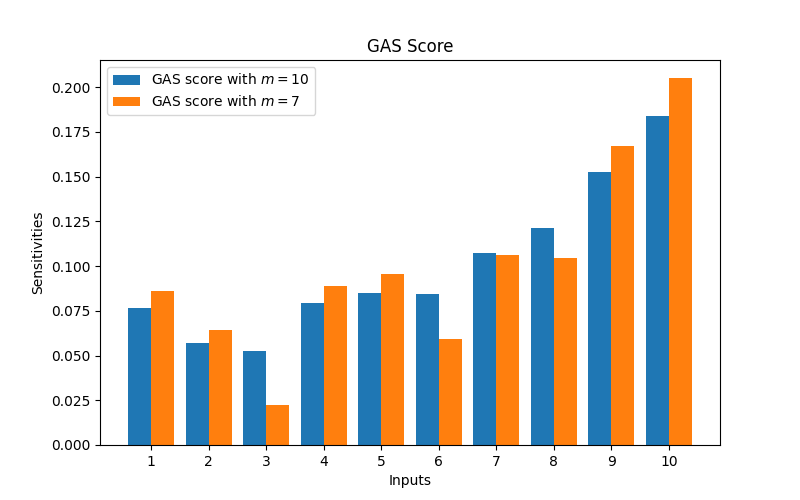}}\\
%\subfloat[Normalized $\mu^*$ of Morris method]{\includegraphics[width=0.5\textwidth]{figs/morris01.png}}\\
\caption{Sensitivity indices when $k=0.1$, in Example 1}
\label{fig:itr_noise_3_5}
\end{figure}

\subsubsection{High noise ($k=1$)}

Fig. \ref{fig:11} plots the eigenvalues and the first eigenvector, and Fig. \ref{fig:itr_noise_4} plots the sensitivity indices for the case of high noise. Based on the eigenvalues, we pick $m=8$ for activity scores, and $m=9$ for global activity scores. DGSM and activity scores fail in identifying the important variables like in the case of moderate noise. We now observe a deterioration in the performance of global activity scores. Although a few pairwise rankings match those of the Sobol' indices, the method assigns substantially more importance to the least important variables according to Sobol'.

%All the inputs have about the same importance according to DGSM, and the least important input becomes the most important according to activity scores. Both global activity scores and activity scores give different ranking from upper Sobol' indices in this case.
% Global activity scores with $m=9$ ranks input parameters the same way as upper Sobol' indices, except for the ranking of inputs 3 and 5. If we increase $m$ to 10, then global activity scores and the upper Sobol' indices 
%and $\mu^*$ of Morris method 

\begin{figure}[h] % Defines figure environment
\centering
\subfloat{\includegraphics[width=0.5\textwidth]{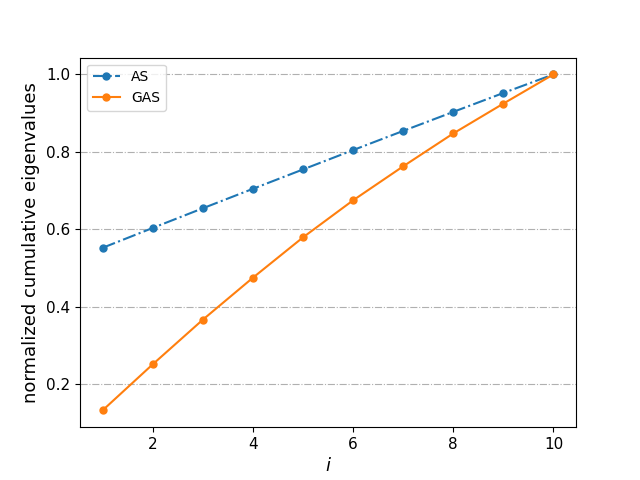}}
\subfloat{\includegraphics[width=0.5\textwidth]{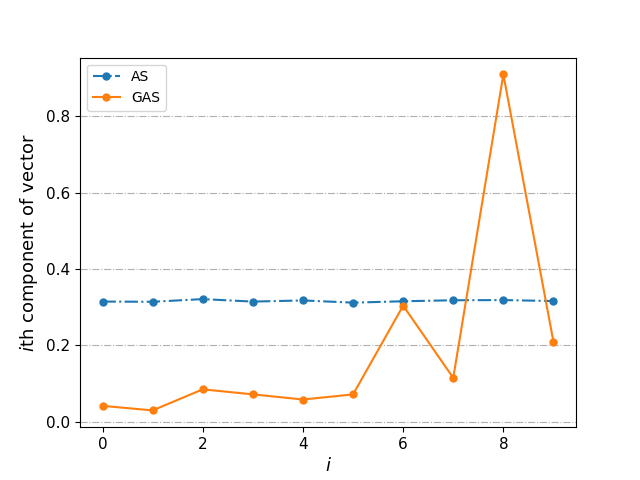}}\\
\caption{Normalized cumulative sum of eigenvalues (left) and the first eigenvector (right) of $\pmb C_\text{as}$ and $\pmb C_\text{gas}$, when $k=1$, in Example 1}
\label{fig:11}
\end{figure}

\begin{figure}[h] % Defines figure environment
\centering
\subfloat[Sobol' indices]{\includegraphics[width=0.5\textwidth]{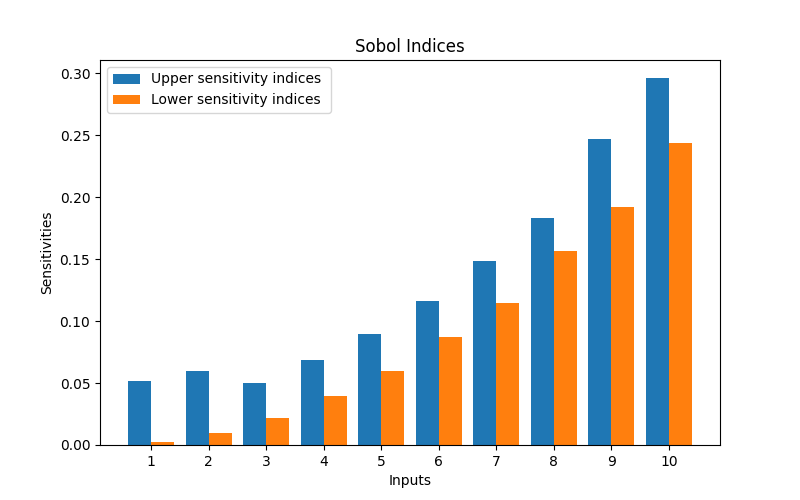}}
\subfloat[Normalized DGSMs]{\includegraphics[width=0.5\textwidth]{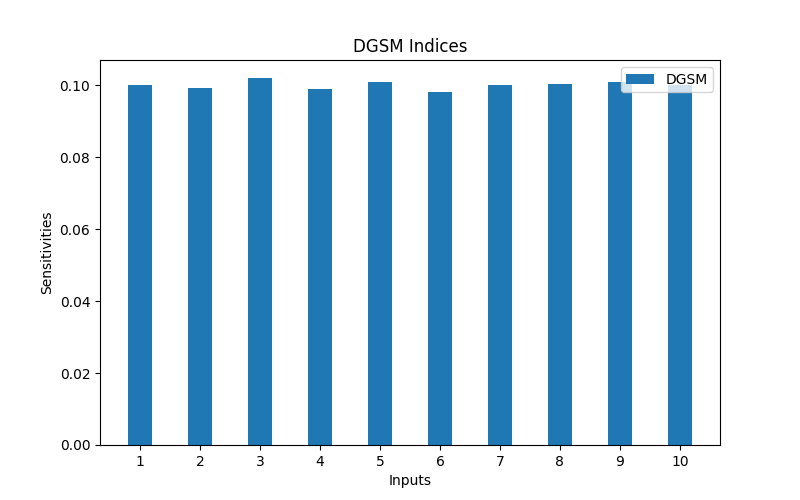}}\\
\subfloat[Normalized activity scores]{\includegraphics[width=0.5\textwidth]{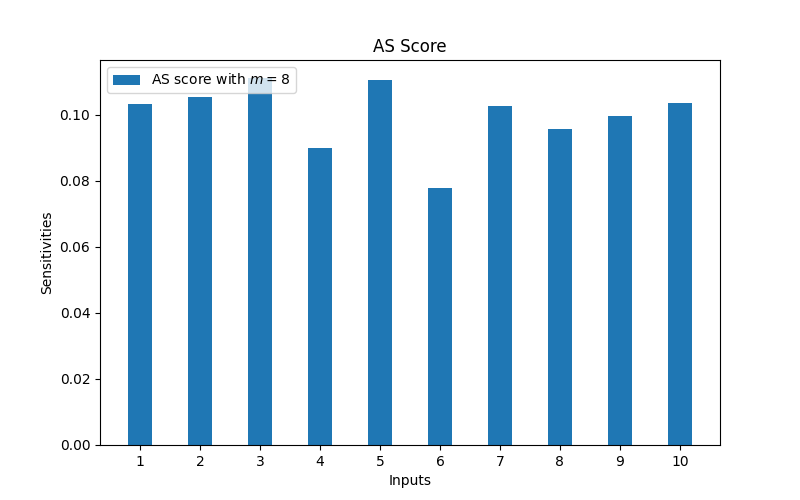}}
\subfloat[Normalized global activity scores]{\includegraphics[width=0.5\textwidth]{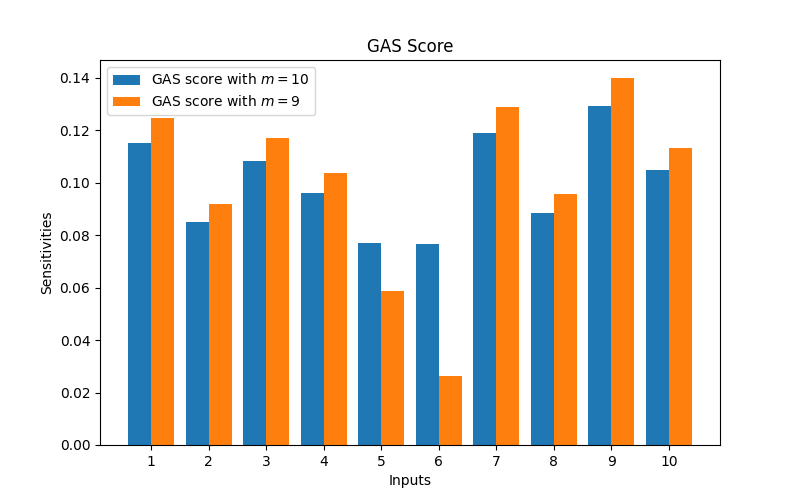}}\\
%\subfloat[Normalized $\mu^*$ of Morris method]{\includegraphics[width=0.5\textwidth]{figs/morris11.png}}\\
\caption{Sensitivity indices when $k=1$, in Example 1}
\label{fig:itr_noise_4}
\end{figure}

\begin{remark}
We also applied the Morris $\mu^*$ elementary-effects measure to Example~1. It is at least as robust as the global activity scores, and the resulting factor ranking is consistent with those obtained using Sobol' indices, although the relative magnitudes of the individual sensitivity measures differ.
\end{remark}

\subsection{Example 2: IEEE 14-bus power system}
In this example we consider the global sensitivity analysis of the IEEE 14-bus power system \cite{Ieee14busbar}. A description of the problem is provided in Yue et al. \cite{yue2022comparison}, who used Sobol' sensitivity indices, DGSMs, and activity scores to analyze the model. Here we will apply the global activity score method to the power system problem and compare it with the other sensitivity measures.

The output of interest in the sensitivity analysis is the steady state solution for the voltage at busbar 9. The inputs are the active power at each node $i$. We set up the differential equations as in Yue et al. \cite{yue2022comparison}, and then find the output by computing the load flow solution with Newton-Raphson method (Tinney and Hart \cite{tinney1967power}). We assume the $i$th input ($i=1,\ldots,14$) follows the distribution $N \left({P}_{i},(0.5 {P}_{i})^{2} \right)$, which corresponds to the largest variance 
uncertainty scenario considered in Ni et al. \cite{ni2017variance}. The value of ${P}_i$'s are presented in Table \ref{table:base}. %(as well as the voltage $V_i$)

\begin{table}[htbp!]
\caption{Means of normal distributions for power system}% and voltages
\centering 
\begin{tabular}{cccccccccc} 
\hline
\textbf{Busbar ${i}$} & \textbf{1}& \textbf{2}& \textbf{3}& \textbf{4}& \textbf{5}& \textbf{6}& \textbf{7}\\
\hline
\textbf{${P}_{i}$}	&0	&21.7	&94.2	&47.8	&7.6	&11.2	&0	 \\ 
%\textbf{$\bar{V}_{i}$}	&1.06	&1.045	&1.01	&1.019	&1.02	&1.07	&1.062\\
\hline
\textbf{Busbar ${i}$} & \textbf{8}& \textbf{9}& \textbf{10}& \textbf{11}& \textbf{12}& \textbf{13}& \textbf{14}\\
\hline
\textbf{${P}_{i}$}	&0	&29.5	&9	&3.5	&6.1	&13.5	&14.9 \\ 
%\textbf{$\bar{V}_{i}$}	&1.09	&1.056	&1.051	&1.057	&1.055	&1.05	&1.036\\
\hline
\end{tabular}
\label{table:base}
%\hline
\end{table}

%Figure \ref{fig:power} plots the normalized cumulative sum of eigenvalues and the first eigenvector of $\pmb C_\text{as}$ and $\pmb C_\text{gas}$. We observe that the first eigenvalue of $\pmb C_\text{gas}$ is much larger than that of $\pmb C_\text{as}$, and the corresponding eigenvectors are quite different. The figure suggests choosing $m=5$ for activity scores, and $m=6$ for global activity scores. 

%\begin{figure}[htpb!] % Defines figure environment
%\centering
%\subfloat{\includegraphics[width=0.5\textwidth]{figs/powerE1.png}}
%\subfloat{\includegraphics[width=0.5\textwidth]{figs/powerE2.png}}\\
%\caption{Normalized cumulative sum of eigenvalues and the first eigenvector of $\pmb C_\text{as}$ and $\pmb C_\text{gas}$ for power system}
%\label{fig:power}
%\end{figure}

In the numerical results, we use a Monte Carlo sample size of $N=10{,}000$ for the Sobol' indices, DGSM, and activity scores. For the global activity scores, we use a Monte Carlo sample size of $N=1{,}000$. 

Figure \ref{fig:1} plots the sensitivity indices. DGSM and activity score methods suggest very different important variables than Sobol' indices. The normalized DGSM and activity scores for nodes 4, 6, 13, and 14 are about zero (the largest value is $3\times 10^{-5}$). However, these nodes have significant importance, according to Sobol' indices. For example, node 4 has the second largest upper Sobol' index with a value of 0.22. The upper Sobol' indices for nodes 6 and 14 are 0.09 and 0.082. The global activity scores, on the other hand, are very similar to upper Sobol' indices. According to upper Sobol' indices, the six most important inputs, in decreasing order, are 3, 4, 9, 13, 6, 14, and according to global activity scores ($m=10$), the most important inputs are 3, 9, 4, 13, 14, 6.

\begin{figure}[htpb!] % Defines figure environment
\centering
\subfloat[Sobol' indices]{\includegraphics[width=0.5\textwidth]{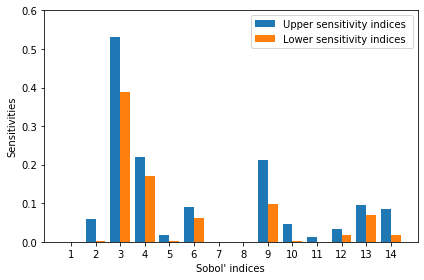}}
\subfloat[Normalized DGSMs]{\includegraphics[width=0.5\textwidth]{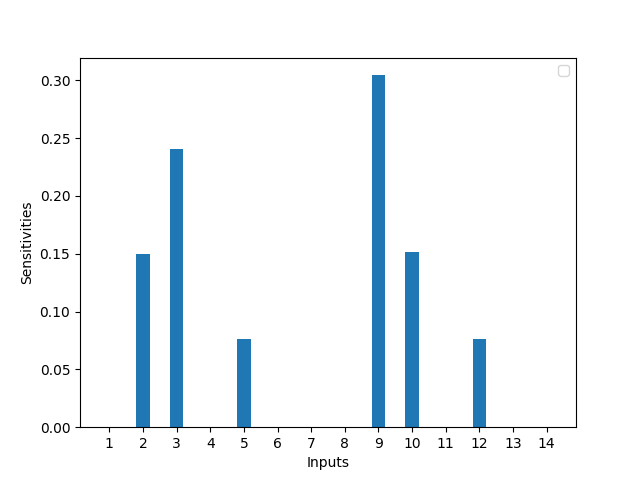}}\\
\subfloat[Normalized activity scores]{\includegraphics[width=0.5\textwidth]{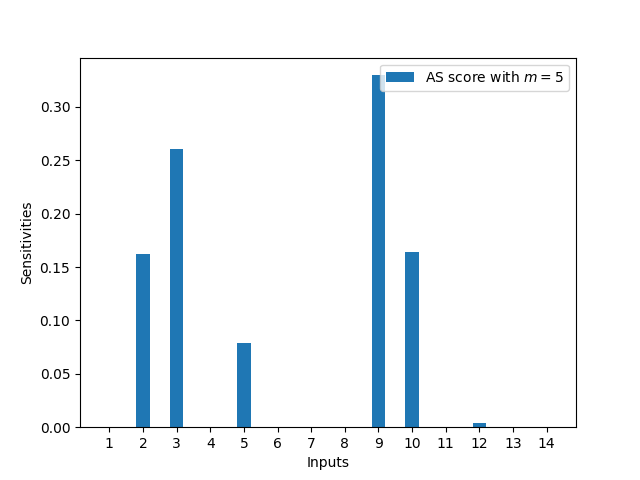}}
\subfloat[Normalized global activity scores]{\includegraphics[width=0.5\textwidth]{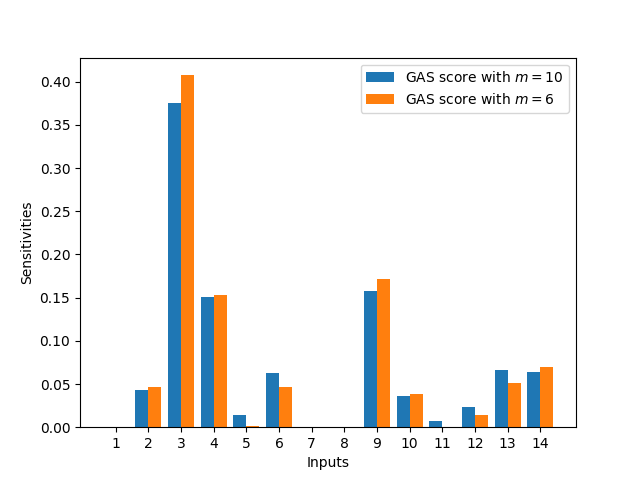}}\\
\caption{Sensitivity indices for power system}
\label{fig:1}
\end{figure}

\subsection{Example 3: A counterexample from Sobol' and Kucherenko}
Sobol' and Kucherenko \cite{kucherenko2010new} present a function where the ranking of important variables by DGSM may result in false conclusions and give a different ranking than the upper Sobol' indices. The function is given by
\[
f(\pmb x)= \sum_{i=1}^4 c_i \left(x_i-1/2\right)+c_{12}\left(x_1-1/2\right) \left(x_2-1/2\right)^5,
\]
where $\pmb x = (x_1,x_2,x_3,x_4), c_i=1,1\leq i \leq 4$ and $c_{12}=50$.
The upper Sobol' index for $x_1$ and $x_2$ is 0.289, and the upper Sobol' index for $x_3$ and $x_4$ is 0.237.

We want to examine how the global activity scores compare with upper Sobol' and DGSM for this example. In the numerical results, we use a Monte Carlo sample size of $N=20{,}000$ for the Sobol' indices, DGSM, and activity scores. For the global activity scores, we use a Monte Carlo sample size of $N=2{,}000$. 
Fig. \ref{fig:countereg} plots the sensitivity indices for the function. Normalized DGSMs and activity scores ($m=2$) give the same results: they find inputs 1, 3, and 4 as comparable in importance, and input 2 as much more important. Sobol' and Kucherenko explain the reason for this discrepancy by the strong nonlinearity of the term $c_{12}\left(x_1-1/2\right) \left(x_2-1/2\right)^5$. 
%Normalized Morris method's $\mu^*$'s give different ranking from Sobol' sensitivity indices as well, that the $\mu^*$ of the first variable is larger than the second one.
The normalized global activity scores rank the first two variables as more important than the last two, similar to the Sobol' sensitivity indices; however, they identify the second input as more important than the first, though to a lesser extent compared to DGSM and activity scores.

\begin{figure}[h] % Defines figure environment
\centering
\subfloat[Sobol' indices]{\includegraphics[width=0.5\textwidth]{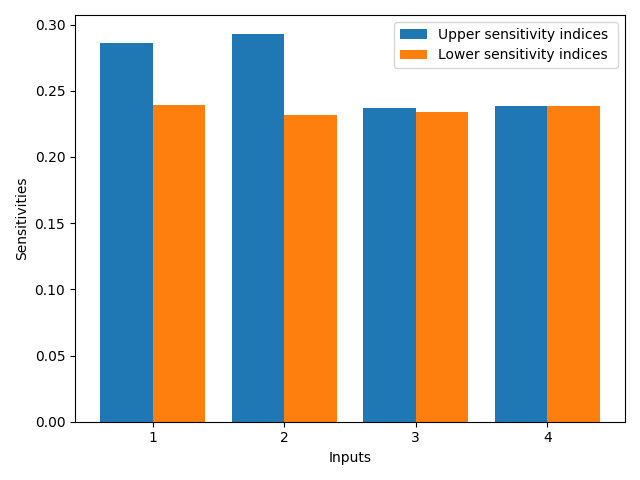}}
\subfloat[Normalized DGSMs]{\includegraphics[width=0.5\textwidth]{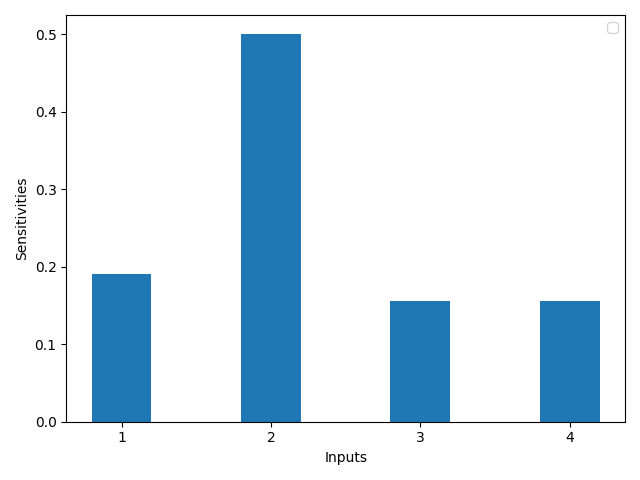}}\\
\subfloat[Normalized activity scores]{\includegraphics[width=0.5\textwidth]{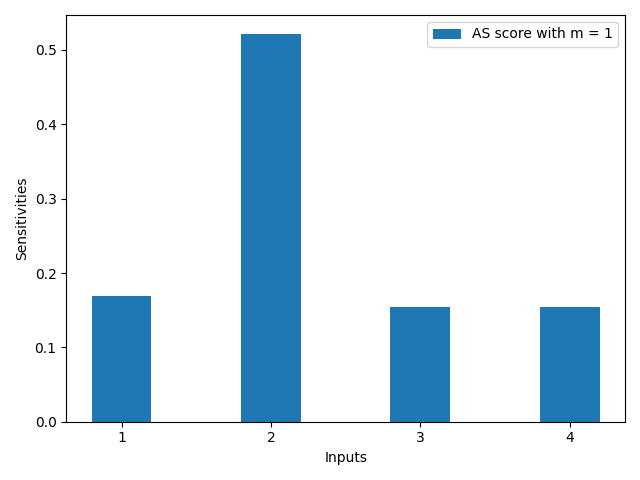}}
\subfloat[Normalized global activity scores]{\includegraphics[width=0.5\textwidth]{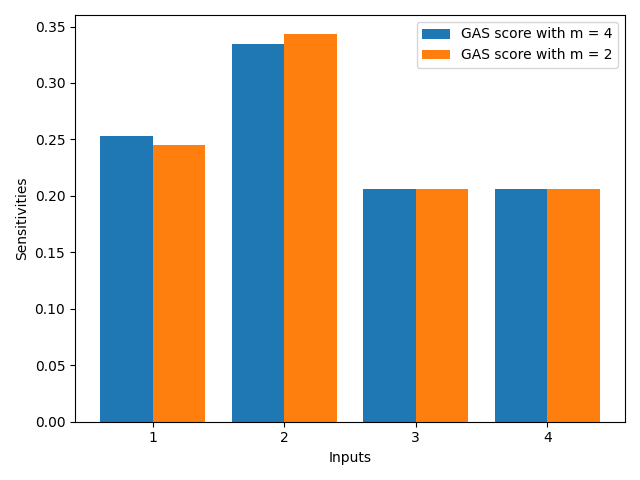}}\\
%\subfloat[Normalized $\mu^*$ of Morris method]{\includegraphics[width=0.5\textwidth]{figs/morris4dim.png}}\\
\caption{Sensitivity indices for the example from Sobol' and Kucherenko}
\label{fig:countereg}
\end{figure}

\section{Conclusions}
\label{sec:conc}
Through the numerical results, we observed that the global activity scores align with the upper Sobol' indices much more closely than DGSM or activity scores in the presence of noise (Example 1), large variance (Example 2), and strong nonlinearity (Example 3). In the absence of such factors, we found empirically that the global activity scores behave similarly to DGSM and activity scores (we report some of these results - specifically the no-noise case in Example 1 - in the paper). These findings highlight the robustness of global activity scores in challenging settings where other sensitivity measures may be misleading. However, we emphasize that global activity scores are not immune to noise, as demonstrated by the high noise setting of Example 1. Beyond the examples presented here, we expect the method to be broadly applicable in problems involving noisy simulations or strong nonlinearities, making it a valuable addition to the global sensitivity analysis toolbox.

Much of the theory of the global activity scores presented in the paper requires independence of the inputs. For problems with dependent inputs, the method can be generalized in two ways. One is to use the approach of Lamboni and Kucherenko \cite{LAMBONI2021107519}, which uses a dependency model that maps the dependent inputs to a vector of independent variables, and compute the global activity scores for the independent variables. The other is to use the approach of Duan and \"{O}kten \cite{duan2023derivative}, which uses global activity scores to define a Shapley value that can then be used for models with dependent inputs.

\appendix

\section{Well-posedness of $\pmb C$}

We present sufficient conditions for the existence of the integrals in $\pmb C$ (Eqn. (\ref{def_C})), for the case of bivariate $f$ -- the arguments generalize to higher dimensions in a straightforward way. 

It suffices to consider the integrals
\begin{equation}\nonumber
\begin{split}
C_{11}&=\int_{\mathbb{R}^3}\frac{(f(v_1,z_2)-f(z_1,z_2))^2}{(v_1-z_1)^2}d\tilde{F}(v_1, z_1,z_2),\\
C_{12}&=\int_{\mathbb{R}^4}\frac{(f(v_1,z_2)-f(z_1,z_2))(f(z_1,v_2)-f(z_1,z_2))}{(v_1-z_1)(v_2-z_2)}d\tilde{F}(v_1,v_2,z_1,z_2),
\end{split}
\end{equation}
where we assume $\tilde{F}$ is an $s$-dimensional Stieltjes measure function (Kingman and Taylor \cite{kingman}) that is absolutely continuous with respect to Lebesgue measure. This assumption is satisfied, for example, when the components $z_i$ and $v_i$ are independent and each has a density, in which case $\tilde{F}$ can be written as the product of marginal distribution functions:
\begin{align}
d\tilde{F}(v_1, z_1,z_2)&=dF_1(v_1)dF_1(z_1)dF_2(z_2),\\
d\tilde{F}(v_1,v_2,z_1,z_2)&=dF_1(v_1)dF_2(v_2)dF_1(z_1)dF_2(z_2).
\end{align}
A Stieltjes measure function defines a measure $\mu_{\tilde{F}}$ on half-open rectangles in $\mathbb{R}^s$. This measure is then extended, using standard measure theory techniques, to Borel sets in $\mathbb{R}^s$, whose completion yields the Lebesgue--Stieltjes measurable sets in $\mathbb{R}^s$.

\begin{theorem}
Let $f(x,y):\mathbb{R}^2 \rightarrow \mathbb{R}$ be a bounded function with continuous and bounded first-order partial derivatives. Define $g(v_1,z_1,z_2)$ by 
\[
g(v_1,z_1,z_2) = \begin{cases}
\left(\frac{f(v_1,z_2)-f(z_1,z_2)}{v_1-z_1}\right)^2 \text{if } v_1 \neq z_1,\\
\left(\frac{\partial f}{\partial x}(z_1,z_2)\right)^2 \text{if } v_1 = z_1.
\end{cases}
\]
Then the integral $C_{11}$, defined as the limit of Riemann-Stieltjes integrals,
\[
C_{11}=\lim_{n\rightarrow \infty} \int_{\mathbb{R}^3\backslash \mathcal{B}_n} g(v_1,z_1,z_2) d\tilde{F}(v_1, z_1,z_2)
\]
exists, where 
\[
\mathcal{B}_n=\{(v_1,z_1,z_2) \in \mathbb{R}^3, \text{s.t. } |v_1-z_1|<1/n \},n>0.
\]
\end{theorem}

\begin{proof}
For each $n>0$, the function $g$ is bounded and continuous on 
$\mathbb{R}^3 \setminus \mathcal{B}_n$, hence the Riemann--Stieltjes integral 
\[
I_n=\int_{\mathbb{R}^3 \setminus \mathcal{B}_n} g(v_1,z_1,z_2)\, d\tilde{F}(v_1, z_1,z_2)
\]
is well-defined.  The sequence $(I_n)$ is monotone increasing, since $\mathcal{B}_{n+1}\subseteq \mathcal{B}_n$, and it is bounded above by 
$
\int_{\mathbb{R}^3} g \, d\tilde{F}.
$
Thus the limit $\lim_{n\to\infty} I_n$ exists.  
Moreover, since $g$ is bounded on $\mathbb{R}^3$, say $0 \leq g \leq C$, we have
\[
\int_{\mathcal{B}_n} g \, d\tilde{F} \leq C\, \mu_{\tilde{F}}(\mathcal{B}_n) \;\longrightarrow\; 0,
\]
as $n\to\infty$. Therefore,
\[
C_{11} = \int_{\mathbb{R}^3} g \, d\tilde{F}.
\]
\end{proof}

\begin{theorem}
Let $f(x,y):\mathbb{R}^2 \rightarrow \mathbb{R}$ be a bounded function with continuous and bounded first-order partial derivatives. Define 
\[
g(v_1,v_2,z_1,z_2) = 
\frac{(f(v_1,z_2)-f(z_1,z_2))(f(z_1,v_2)-f(z_1,z_2))}{(v_1-z_1)(v_2-z_2)},
\]
if $v_1 \neq z_1$ and $v_2 \neq z_2$. If $v_1 = z_1$, then the corresponding term 
$\frac{f(v_1,z_2)-f(z_1,z_2)}{v_1-z_1}$ is replaced by $\frac{\partial f}{\partial x}(z_1,z_2)$. 
The case $v_2 = z_2$ is handled similarly.  
Then the integral $C_{12}$, defined as the limit of Riemann--Stieltjes integrals,
\[
C_{12}=\lim_{n\rightarrow \infty} \int_{ \mathbb{R}^4\backslash \mathcal{B}_n} g(v_1,v_2,z_1,z_2)\, d\tilde{F}(v_1,v_2,z_1,z_2),
\]
exists, where, for $n>0$,
\[
\mathcal{B}_n=\{(v_1,v_2,z_1,z_2) \in \mathbb{R}^4 : |v_1-z_1|<1/n \}\,\cup\,\{(v_1,v_2,z_1,z_2) \in \mathbb{R}^4 : |v_2-z_2|<1/n \}.
\]
\end{theorem}

\begin{proof}
Since $g$ is bounded and continuous on $\mathbb{R}^4\backslash \mathcal{B}_n$ for all $n>0$,
the Riemann-Stieltjes integrals $\int_{ \mathbb{R}^4\backslash \mathcal{B}_n} g(v_1,v_2,z_1,z_2)$ exist.
Let
\[
I_n=\int_{ \mathbb{R}^4\backslash \mathcal{B}_n} g(v_1,v_2,z_1,z_2)\, d\tilde{F}(v_1,v_2,z_1,z_2),
\]
and
\[
J_n=\int_{ \mathbb{R}^4\backslash \mathcal{B}_n} |g(v_1,v_2,z_1,z_2)|\, d\tilde{F}(v_1,v_2,z_1,z_2). 
\]
Then $(J_n)$ is an increasing sequence bounded by $\int_{ \mathbb{R}^4} |g|\, d\tilde{F}$, so $\lim_{n\to\infty} J_n$ exists.  
Moreover,
\[
0\leq |g(v_1,v_2,z_1,z_2)|-g(v_1,v_2,z_1,z_2)\leq 2 |g(v_1,v_2,z_1,z_2)|,
\]
which implies
\[
\lim_{n\rightarrow \infty} \int_{ \mathbb{R}^4\backslash \mathcal{B}_n} \big(|g(v_1,v_2,z_1,z_2)|-g(v_1,v_2,z_1,z_2)\big)\,d\tilde{F}(v_1,v_2,z_1,z_2)
\]
exists, since the sequence is increasing and bounded by $2\int_{ \mathbb{R}^4} |g| d\tilde{F}$.  
Since the limits of the integrals of $|g|$ and $|g|-g$ both exist, the limit
\[
C_{12}=\lim_{n\rightarrow \infty} I_n
\]
exists as well.  
Finally, since $g$ is bounded on $\mathbb{R}^4$,  
$\int_{\mathcal{B}_n} g d\tilde{F} \rightarrow 0$ as $n\rightarrow\infty$,
and hence
\[
C_{12}=\int_{ \mathbb{R}^4} g(v_1,v_2,z_1,z_2)\, d\tilde{F}.
\]
\end{proof}

\section{Computational cost of sensitivity indices}
\label{appendixB}

Below we present tables summarizing the number of function evaluations and computational time\footnote {In examples 1 and 3, numerical simulations and sensitivity analysis experiments were conducted on a 2022 MacBook Pro workstation. The system is equipped with an Apple M2 System on a Chip (SoC) featuring an 8-core CPU based on the ARM64 architecture, and 8 GB of unified memory. The operating system used was macOS Sonoma (Version 14.4.1).
In example 2, numerical experiments were conducted using MATLAB Online (R2025b) hosted on a high-performance AWS cloud instance. The underlying hardware featured an Intel Xeon Platinum 8488C CPU (16 vCPUs). The system possessed a total physical memory of 128 GB, which was dynamically allocated as a shared resource.} for the examples considered in Section \ref{sec:numerical}. Table \ref{tab:cost_1} presents the results for the no-noise setting of Example 1. The other noisy cases yield similar results. Tables \ref{tab:cost_2} and \ref{tab:cost_3} present the results for Examples 2 and 3.

\begin{table}[htbp]
\centering
\begin{tabular}{lcc} 
\hline
\textbf{Sensitivity index} & \textbf{Function evaluations} & \textbf{Time (s)}\\
\hline
Upper Sobol'           & $220{,}000$   & 0.074 \\
Lower Sobol'           & $440{,}000$   & 0.141 \\
Activity scores        & $220{,}000$   & 0.055 \\
DGSM                   & $220{,}000$   & 0.055 \\
Global activity scores & $202{,}000$   & 0.075 \\
%$\mu^*$ of Morris method & $220,000$  & 0.063\\ 
\hline
\end{tabular}
\caption{Comparison of computational cost: Example 1, $k = 0$}
\label{tab:cost_1}
\end{table}

\begin{table}[htbp]
\centering
\begin{tabular}{lcc} 
\hline
\textbf{Sensitivity index} & \textbf{Function evaluations} & \textbf{Time (s)}\\
\hline
Upper Sobol'           & $150{,}000$   & 245 \\
Lower Sobol'           & $300{,}000$   & 528 \\
Activity scores        & $150{,}000$   & 254 \\
DGSM                   & $150{,}000$   & 254 \\
Global activity scores & $141{,}000$   & 261 \\
\hline
\end{tabular}
\caption{Comparison of computational cost: Example 2}
\label{tab:cost_2}
\end{table}

\begin{table}[htbp]
\centering
\begin{tabular}{lcc} 
\hline
\textbf{Sensitivity index} & \textbf{Function evaluations} & \textbf{Time (s)}\\
\hline
Upper Sobol'           & $100{,}000$   & 0.014 \\
Lower Sobol'           & $200{,}000$  & 0.035 \\
Activity scores        & $100{,}000$   & 0.029 \\
DGSM                   & $100{,}000$   & 0.029 \\
Global activity scores & $102{,}000$  & 0.010 \\
%$\mu^*$ of Morris method & $100,000$  & 0.037\\ 
\hline
\end{tabular}
\caption{Comparison of computational cost: Example 3}
\label{tab:cost_3}
\end{table}

\bibliographystyle{siamplain}
\bibliography{references}
\end{document}

%% file: shared.tex
\usepackage{lineno}
\usepackage{lipsum}
\usepackage{amsfonts}
\usepackage{graphicx,epstopdf}
\usepackage{epstopdf}
\usepackage{algorithmic}
\usepackage{multirow}
\usepackage{subfig}
\usepackage{enumerate}
\usepackage{algorithm}
\usepackage{graphicx}
\usepackage{textcomp}

\ifpdf
  \DeclareGraphicsExtensions{.eps,.pdf,.png,.jpg}
\else
  \DeclareGraphicsExtensions{.eps}
\fi

\newsiamremark{remark}{Remark}
\newsiamremark{hypothesis}{Hypothesis}
\crefname{hypothesis}{Hypothesis}{Hypotheses}
\newsiamthm{claim}{Claim}
\headers{Global Active Subspace}{Ruilong Yue, Giray \"{O}kten}

\title{Global Activity Scores\thanks{Submitted to the editors DATE.
}}

\author{Ruilong Yue\thanks{Department of Mathematics, Florida State University, Tallahassee, FL 
  (\email{ryue@math.fsu.edu}).}
\and Giray \"{O}kten\thanks{Corresponding author. Department of Mathematics, Florida State University, Tallahassee, FL 
  (\email{okten@math.fsu.edu}).}}

\usepackage{amsopn}

%% file: references.bib
@preamble{"\ifx \manfnt \undefined \font\manfnt=logo10 \fi" #
   "\ifx \METAFONT \undefined \def \METAFONT {{\manfnt META}\-{\manfnt FONT}\spacefactor1000 } \fi" #
   "\ifx \MF \undefined \let \MF=\METAFONT \fi" #
   "\ifx \POSTSCRIPT \undefined \def \POSTSCRIPT {{\scshape Post}\-{\scshape Script}\spacefactor1000 } \fi" #
   "\ifx \MP \undefined \def \MP {{\manfnt META}\-{\manfnt POST}\spacefactor1000 } \fi" #
   "\ifx \noopsort \undefined \def \noopsort#1{} \fi" #
   "\ifx \emdash \undefined \def \emdash{---} \fi"}

@article{sobol1993sensitivity,
	author = {Sobol', I. M.},
	journal = {Mathematical Modelling and Computational Experiment},
	number = {4},
	pages = {407--414},
	title = {Sensitivity analysis for non-linear mathematical models},
	volume = {1},
	year = {1993}}

@article{tinney1967power,
	author = {W. F. Tinney and C. E. Hart},
	journal = {IEEE Transactions on Power Apparatus and Systems},
	number = {11},
	pages = {1449--1460},
	title = {Power Flow Solution by {N}ewton's Method},
	volume = {PAS-86},
	year = {1967}}

@article{LAMBONI2021107519,
	author = {Matieyendou Lamboni and Sergei Kucherenko},
	issn = {0951-8320},
	journal = {Reliability Engineering \& System Safety},
	pages = {107519},
	title = {Multivariate sensitivity analysis and derivative-based global sensitivity measures with dependent variables},
	volume = {212},
	year = {2021}}

@article{roustant2017poincare,
	author = {Roustant, Olivier and Barthe, Franck and Iooss, Bertrand},
	journal = {Electronic Journal of Statistics},
	number = {2},
	pages = {3081--3119},
	title = {Poincar{\'e} inequalities on intervals--application to sensitivity analysis},
	volume = {11},
	year = {2017}}

@article{lamboni2013derivative,
	author = {Lamboni, Matieyendou and Iooss, Bertrand and Popelin, A-L and Gamboa, Fabrice},
	journal = {Mathematics and Computers in Simulation},
	pages = {45--54},
	publisher = {Elsevier},
	title = {Derivative-based global sensitivity measures: General links with {S}obol' indices and numerical tests},
	volume = {87},
	year = {2013}}

@book{kingman,
	author = {Kingman, J. F. C. and Taylor, S. J.},
	publisher = {Cambridge University Press},
	title = {Introduction to Measure and Probability},
	year = {1966}}

@article{ni2017variance,
	author = {Ni, Fei and Nijhuis, Michiel and Nguyen, Phuong H and Cobben, Joseph FG},
	journal = {IEEE Transactions on Power Systems},
	number = {2},
	pages = {1670--1682},
	publisher = {IEEE},
	title = {Variance-based global sensitivity analysis for power systems},
	volume = {33},
	year = {2017}}

@article{constantine2014active,
	author = {Constantine, Paul G and Dow, Eric and Wang, Qiqi},
	journal = {SIAM Journal on Scientific Computing},
	number = {4},
	pages = {A1500--A1524},
	publisher = {SIAM},
	title = {Active subspace methods in theory and practice: applications to kriging surfaces},
	volume = {36},
	year = {2014}}

@article{constantine2017global,
	author = {Constantine, Paul G and Diaz, Paul},
	journal = {Reliability Engineering \& System Safety},
	pages = {1--13},
	publisher = {Elsevier},
	title = {Global sensitivity metrics from active subspaces},
	volume = {162},
	year = {2017}}

@inproceedings{yue2022comparison,
	author = {Yue, Ruilong and Duan, Hui and {\"O}kten, Giray and Uzuno{\u{g}}lu, Bahri},
	booktitle = {2022 6th International Conference on System Reliability and Safety (ICSRS)},
	organization = {IEEE},
	pages = {130--137},
	title = {A Comparison of Global Sensitivity Methods for Power Systems},
	year = {2022}}

@article{sobol2010derivative,
	author = {Sobol', IM and Kucherenko, S},
	journal = {Procedia Social and Behavioral Sciences},
	number = {6},
	pages = {7745--7746},
	publisher = {Elsevier},
	title = {Derivative based global sensitivity measures},
	volume = {2},
	year = {2010}}

@article{duan2023derivative,
	author = {Duan, Hui and {\"O}kten, Giray},
	journal = {International Journal for Uncertainty Quantification},
	number = {1},
	pages = {1-16},
	title = {Derivative-based {S}hapley value for global sensitivity analysis and machine learning explainability},
	volume = {15},
	year = {2025}}

@article{sobol2009derivative,
	author = {Sobol', I.M and Kucherenko, S},
	journal = {Mathematics and Computers in Simulation},
	number = {10},
	pages = {3009--3017},
	publisher = {Elsevier},
	title = {Derivative based global sensitivity measures and their link with global sensitivity indices},
	volume = {79},
	year = {2009}}

@article{sobol2001global,
	author = {Sobol', Ilya M},
	journal = {Mathematics and {C}omputers in {S}imulation},
	number = {1-3},
	pages = {271--280},
	publisher = {Elsevier},
	title = {Global sensitivity indices for nonlinear mathematical models and their {M}onte {C}arlo estimates},
	volume = {55},
	year = {2001}}

@book{saltelli2008global,
	author = {Saltelli, Andrea and Ratto, Marco and Andres, Terry and Campolongo, Francesca and Cariboni, Jessica and Gatelli, Debora and Saisana, Michaela and Tarantola, Stefano},
	publisher = {John Wiley \& Sons},
	title = {Global sensitivity analysis: the primer},
	year = {2008}}

@article{saltelli2010variance,
	author = {Saltelli, Andrea and Annoni, Paola and Azzini, Ivano and Campolongo, Francesca and Ratto, Marco and Tarantola, Stefano},
	journal = {Computer Physics Communications},
	pages = {259--270},
	title = {Variance based sensitivity analysis of model output. {D}esign and estimator for the total sensitivity index},
	volume = {181},
	year = {2010}}

@book{constantine2015active,
	author = {Constantine, Paul G},
	publisher = {SIAM},
	title = {Active subspaces: Emerging ideas for dimension reduction in parameter studies},
	year = {2015}}

@article{liu2023preintegration,
	author = {Liu, Sifan and Owen, Art B},
	journal = {SIAM Journal on Numerical Analysis},
	number = {2},
	pages = {495--514},
	publisher = {SIAM},
	title = {Preintegration via active subspace},
	volume = {61},
	year = {2023}}

@article{sobol2011derivative,
	author = {Sobol', IM},
	journal = {Mathematical Models and Computer Simulations},
	pages = {419--423},
	publisher = {Springer},
	title = {On derivative-based global sensitivity criteria},
	volume = {3},
	year = {2011}}

@article{kucherenko2014derivative,
	author = {Kucherenko, Serge and Iooss, Bertrand},
	journal = {arXiv preprint arXiv:1412.2619},
	title = {Derivative based global sensitivity measures},
	year = {2014}}

@incollection{kucherenko2017derivative,
	author = {Kucherenko, S and Iooss, B},
	booktitle = {Handbook of Uncertainty Quantification},
	pages = {1241--1263},
	title = {Derivative-based global sensitivity measures},
	year = {2017}}

@article{kucherenko2010new,
	author = {Kucherenko, S and others},
	journal = {Computer Physics Communications},
	number = {7},
	pages = {1212--1217},
	publisher = {Elsevier},
	title = {A new derivative based importance criterion for groups of variables and its link with the global sensitivity indices},
	volume = {181},
	year = {2010}}

@article{campolongo2007effective,
	author = {Campolongo, Francesca and Cariboni, Jessica and Saltelli, Andrea},
	journal = {Environmental Modelling \& Software},
	number = {10},
	pages = {1509--1518},
	publisher = {Elsevier},
	title = {An effective screening design for sensitivity analysis of large models},
	volume = {22},
	year = {2007}}

@article{kucherenko2012estimation,
	author = {Kucherenko, Sergei and Tarantola, Stefano and Annoni, Paola},
	journal = {Computer Physics Communications},
	number = {4},
	pages = {937--946},
	publisher = {Elsevier},
	title = {Estimation of global sensitivity indices for models with dependent variables},
	volume = {183},
	year = {2012}}

@article{iooss2015review,
	author = {Iooss, Bertrand and Lema{\^\i}tre, Paul},
	journal = {Uncertainty Management in Simulation-Optimization of Complex Systems: Algorithms and Applications},
	pages = {101--122},
	publisher = {Springer},
	title = {A review on global sensitivity analysis methods},
	year = {2015}}

@article{hart2017efficient,
	author = {Hart, Joseph L and Alexanderian, Alen and Gremaud, Pierre A},
	journal = {SIAM Journal on Scientific Computing},
	number = {4},
	pages = {A1514--A1530},
	publisher = {SIAM},
	title = {Efficient computation of {S}obol' indices for stochastic models},
	volume = {39},
	year = {2017}}

@article{nanty2016sampling,
	author = {Nanty, Simon and Helbert, C{\'e}line and Marrel, Amandine and P{\'e}rot, Nadia and Prieur, Cl{\'e}mentine},
	journal = {SIAM/ASA Journal on Uncertainty Quantification},
	number = {1},
	pages = {636--659},
	publisher = {SIAM},
	title = {Sampling, metamodeling, and sensitivity analysis of numerical simulators with functional stochastic inputs},
	volume = {4},
	year = {2016}}

@article{fort2021global,
	author = {Fort, Jean-Claude and Klein, Thierry and Lagnoux, Agn{\`e}s},
	journal = {SIAM/ASA Journal on Uncertainty Quantification},
	number = {2},
	pages = {880--921},
	publisher = {SIAM},
	title = {Global sensitivity analysis and {W}asserstein spaces},
	volume = {9},
	year = {2021}}

@techreport{Ieee14busbar,
	institution = {Illinois Center for a Smarter Electric Grid},
	title = {{IEEE} 14-{B}us {S}ystem},
	url = {https://icseg.iti.illinois.edu/ieee-14-bus-system/}}

@article{owen2013better,
	author = {Owen, Art B},
	journal = {ACM Transactions on Modeling and Computer Simulation (TOMACS)},
	number = {2},
	pages = {1--17},
	publisher = {ACM New York, NY, USA},
	title = {Better estimation of small {S}obol' sensitivity indices},
	volume = {23},
	year = {2013}}
